\documentclass[12pt]{amsart}
\usepackage[a4paper, hmargin=2.5cm, vmargin={3.7cm, 3.7cm}]{geometry}
\usepackage[numbers,sort]{natbib}
\bibpunct[, ]{[}{]}{,}{n}{,}{,}% Citation support using natbib.sty
\usepackage{amsmath, amsthm, amssymb, amsfonts, mathtools,comment}
\usepackage{mathrsfs}
\usepackage{hyperref}
\usepackage[ruled,vlined]{algorithm2e}
\usepackage{algpseudocode}
\usepackage{enumitem}
\usepackage[colorinlistoftodos]{todonotes}

\theoremstyle{plain}% Theorem-like structures provided by amsthm.sty
\newtheorem{theorem}{Theorem}[section]
\newtheorem{lemma}[theorem]{Lemma}

\usepackage{float}

\usepackage{graphicx}
\usepackage{subcaption}
\graphicspath{ {Images/} }

\theoremstyle{definition}

\theoremstyle{remark}

\usepackage{xcolor}

\SetCommentSty{mycommfont}

\SetKwInput{KwInput}{Input}                % Set the Input
\SetKwInput{KwFind}{find}              % find function
\SetKwInput{KwLet}{let}              % let
\SetKwInput{KwCompute}{compute}              % compute
\SetKwInput{KwReturn}{return}              % return

\usepackage{wasysym}

\newtheorem*{theorem*}{Theorem}

\newcommand{\sfA}{\mathsf{A}}
\newcommand{\sfB}{\mathsf{B}}
\newcommand{\sfC}{\mathsf{C}}
\newcommand{\sfE}{\mathsf{E}}
\newcommand{\sfH}{\mathsf{H}}
\newcommand{\sfL}{\mathsf{L}}

\newcommand{\cC}{\mathcal{C}}

\newcommand{\cW}{\mathcal{W}}

\newcommand{\C}{\mathbb{C}}
\newcommand{\D}{\mathbb{D}}
\newcommand{\N}{\mathbb{N}}
\newcommand{\R}{\mathbb{R}}
\newcommand{\T}{\mathbb{T}}
\newcommand{\Capacity}{\operatorname{Cap}}

\newcommand{\boldt}{\boldsymbol{t}}
\newcommand{\boldd}{\boldsymbol{d}}
\newcommand{\boldz}{\boldsymbol{z}}
\newcommand{\boldr}{\boldsymbol{r}}
\newcommand{\boldalpha}{\boldsymbol{\alpha}}
\newcommand{\boldlambda}{\boldsymbol{\lambda}}
\newlist{steps}{enumerate}{1}
\setlist[steps, 1]{label = Step \arabic*:}

\newcommand{\re}{\operatorname{Re}}

\newcommand{\red}[1]{#1}

\begin{document}

% \articletype{ARTICLE TEMPLATE}% Specify the article type or omit as appropriate

\title{Computing Chebyshev polynomials using the complex Remez algorithm}

\author{Lennart Aljoscha H\"{u}bner}
\address{Department of Mathematics,
KU Leuven,
Celestijnenlaan 200B,
3001 Leuven, Belgium}
\email{lennart.huebner@kuleuven.be}

\author{Olof Rubin}
\address{Department of Mathematics,
KU Leuven,
Celestijnenlaan 200B,
3001 Leuven, Belgium}
\email{olof.rubin@kuleuven.be}

\maketitle
\begin{abstract}
We employ the generalized Remez algorithm, initially suggested by P. T. P. Tang, to perform an experimental study of Chebyshev polynomials in the complex plane. Our focus lies particularly on the examination of their norms and zeros. What sets our study apart is the breadth of examples considered, coupled with the fact that the degrees under investigation are substantially higher than those in previous studies where other methods have been applied. These computations of Chebyshev polynomials of high degrees reveal discernible, repeating patterns, which indicate a typical behavior of Chebyshev polynomials in a general setting. The use of Tang's algorithm allows for computations executed with precision, maintaining accuracy within quantifiable margins of error. Additionally, as a result of our experimental study, we propose what we believe to be a fundamental relationship between Chebyshev and Faber polynomials associated with a compact set.
\end{abstract}

\section{Introduction}
\label{sec:intro}
Let $\sfE$ be a compact subset of the complex plane $\C$. Our focus is directed towards monic polynomials that exhibit minimal deviation from zero over the set $\sfE$. In other words, for any given positive integer $n$ we want to find the coefficients $a_0^\ast,\dotsc,a_{n-1}^\ast$ satisfying
\begin{equation}
    \max_{z\in \sfE}\left|z^n+\sum_{k=0}^{n-1}a_k^\ast z^k\right| = \min_{a_0,\dotsc,a_{n-1}\in \C}\max_{z\in \sfE}\left|z^n+\sum_{k=0}^{n-1}a_kz^k\right|.
    \label{eq:cheb_definition}
\end{equation}
The existence of minimizing coefficients $a_0^\ast,\dotsc,a_{n-1}^\ast$ is guaranteed through a compactness argument. However, such a minimizer does not need to be unique. If $\sfE$ is \red{finite and consists of} $m<n$ points\red{,} then there are \red{infinitely many different minimizing polynomials}. This is the only exceptional case, the assumption that $\sfE$ consists of infinitely many points ensures the uniqueness of a monic \red{$n$th degree} minimizer of \eqref{eq:cheb_definition} for any $n\in \mathbb{N}$. This is the so-called Chebyshev polynomial of degree $n$ corresponding to the set $\sfE$ \red{which we denote by $T_n^{\sfE}$}. For basic theory detailing the proofs of existence and uniqueness of Chebyshev polynomials we refer the reader to \red{\cite{achieser56,cheney66,christiansen-simon-zinchenko-I, novello-schiefermayr-zinchenko21, rivlin90, iske18,smirnov-lebedev68}}.

Throughout this text we reserve the notation $\|\cdot\|_{\sfE}$ to denote the maximum norm on $\sfE$ and let $\D$ denote the open unit disk and $\T$ the unit circle.

Historically\red{,} the consideration of polynomial minimizers with respect to the maximum norm originates from the studies of P.~L.~Chebyshev who considered minimization on $[-1,1]$, see \cite{chebyshev54}. Chebyshev polynomials corresponding to real sets have been much better understood than the corresponding complex ones. The reason for this discrepancy in understanding can partially be attributed to the powerful alternation theorem\red{,} which is valid for real Chebyshev polynomials, \cite[p. 75]{cheney66}. For any compact set $\sfE\subset \R$ containing at least $n+1$ points\red{,} the Chebyshev polynomial $T_n^{\sfE}$ is characterized by having an alternating set on $\sfE$ consisting of $n+1$ points. That is to say, there are points $x_0<x_1<\cdots <x_n$ all contained in $\sfE$ such that 
\begin{equation}
    T_n^{\sfE}(x_k) = (-1)^{n-k}\|T_n^{\sfE}\|_{\sfE}.
    \label{eq:alternation_theorem}
\end{equation}
This alternating property, whose analogue can be shown for \red{any} real \red{best approximations}, constitutes the theoretical grounding for the classical Remez algorithm which is used to compute real-valued best approximations, see \cite{cheney66,iske18, remez34-1, remez34-2}.

\subsection{Chebyshev polynomials in the complex plane}
\red{Crucially}, the alternation property \red{fails to characterize} Chebyshev polynomials for general complex sets $\sfE\subset \C$. Apart from the fact that the argument of a Chebyshev polynomial at an extremal point can be any angle, not just $k\pi$ with $k\in \mathbb{Z}$, the number of extremal points corresponding to $T_n^{\sfE}$ on $\sfE$ can vary greatly. While there are at least $n+1$ such extremal points on $\sfE$, see e.g. \cite[Theorem 1, p. 446]{smirnov-lebedev68}, there is no upper bound on the number of extremal points. Indeed, as the example $T_n^{\T}(z) = z^n$ shows, the entire sets may consist of extremal points of the Chebyshev polynomial. 

One approach to studying Chebyshev polynomials in the complex plane comes from the fruitful interplay between approximation theory and potential theory. \red{We recall that the logarithmic capacity is a quantity associated with any compact set in the complex plane, see \cite[\S 5.1]{ransford95}. Denoting it with $\Capacity$, }Szeg\H{o} \cite{szego24} proved that
\begin{equation}
    \|T_n^{\sfE}\|_{\sfE}\geq \Capacity(\sfE)^n.
    \label{eq:szego_inequality}
\end{equation}
A recent proof of this fundamental inequality can be found in \cite[Theorem 5.5.4]{ransford95}. Since the capacity and radius of a disk \red{coincide,} this provides an easy way of seeing that $T_n^{\T}(z) = z^n$. If $P(z) = a_mz^m+a_{m-1}z^{m-1}+\cdots+a_0$ is a polynomial of exact degree $m$, then \cite[Theorem 5.2.5]{ransford95} says that
\begin{equation}
    \Capacity\Big(P^{-1}(\sfE)\Big) = \left(\frac{\Capacity(\sfE)}{|a_m|}\right)^{1/m}.
    \label{eq:capacity_preimage}
\end{equation}
\red{If we additionally assume that} $P$ is a monic polynomial of degree $m$ and \red{define the filled-in lemniscate} $\sfE_P = \{z: |P(z)|\leq r\}$, then we gather as a consequence of \eqref{eq:szego_inequality}, \eqref{eq:capacity_preimage} and the uniqueness of Chebyshev polynomials that
\begin{equation}
    T_{nm}^{\sfE_P}(z) = P(z)^n.
    \label{eq:cheb_nm_sequence}
\end{equation}

This example, whose \red{origin} can be traced back to Faber \cite{faber20}, constitutes one of the few cases where the Chebyshev polynomials are determined for certain degrees. In general, for a given compact set $\sfE$, it is rarely the case that the Chebyshev polynomials $T_n^{\sfE}$ admit explicit representations. For instance, the Chebyshev polynomials corresponding to $\{z: |P(z)|\leq r\}$ of degrees other than multiples of $\deg(P)$, remain unknown in the general case. 

Chebyshev polynomials appear in various applications. The classical Chebyshev polynomials\red{,} \red{which minimize the supremum norm on the interval $[-1,1]$,} are fundamental for numerical analysis and approximation theory. This is, to a large extent, due to their relation with Fourier analysis. Chebyshev polynomials on unions of intervals further appear as discriminants corresponding to Jacobi matrices\red{,} which in turn are related to periodic Schr\"{o}dinger operators, see \cite[\S 2]{christiansen-simon-zinchenko-I}. 

The \red{extension} of Chebyshev polynomials to complex sets can also be motivated by \red{their} applicability. \red{For example, \cite{greenbaum-trefethen-94} explains how matrix-valued Chebyshev polynomials have applications} to \red{iterative Krylov subspace methods,} such as the Arnoldi iteration\red{,} which is used to estimate eigenvalues of matrices. Such potential links are further considered in \cite{toh-trefethen98, faber-liesen-tichy10}. If the matrix in question is normal then the matrix-valued Chebyshev polynomials coincide with the Chebyshev polynomials on the spectrum of the \red{corresponding} matrix. 

\red{Closely related, residual Chebyshev polynomials are also minimizers of the supremum norm on a compact set, but instead of being monic they are normalized to attain the value $1$ at some specified point. When the reference point is at infinity this normalization should be understood as fixing the leading coefficient to be $1$ and hence the Chebyshev polynomials emerge. This modification gives rise to differences but many properties are shared.} Residual matrix-valued Chebyshev polynomials appear \red{while estimating the convergence rate} of the GMRES algorithm, see \cite{greenbaum-trefethen-94}.  For theoretical aspects of \red{residual Chebyshev polynomials} see \cite{christiansen-simon-zinchenko-V}. \red{While we will focus on Chebyshev polynomials, the numerical methods presented here can be applied to the residual case as well.}

\red{We believe these examples indicate that determining Chebyshev polynomials is interesting for various reasons} and not limited to understanding fundamental properties of approximation theory.

\subsection{Two different approaches}

To remedy the fact that Chebyshev polynomials \red{typically lack explicit formulas}, one common approach to \red{understand} their asymptotic behavior is to compare them to \red{other} classes of polynomials. One such class of polynomials \red{is} the Faber polynomials \cite{faber20}. If $\sfE\subset \C$ is a simply connected compact set which consists of more than one point, \red{then} there exists a conformal mapping $\Phi:\C\setminus \sfE\rightarrow \C\setminus \overline{\D}$ of the form \begin{equation}
    \Phi(z) = \Capacity(\sfE)^{-1}z+a_0+a_{-1}z^{-1}+\cdots,
    \label{eq:exterior-conformal}
\end{equation}
see \cite[Chapter 2]{smirnov-lebedev68}. The \red{(normalized) $n$th} Faber polynomial, denoted $F_n^\sfE$, is the monic polynomial of degree $n$ defined by the equation
\begin{equation}
    \Big(\Capacity(\sfE)\Phi(z)\Big)^n = F_n^\sfE(z)+O(z^{-1}),\quad z\rightarrow \infty.
    \label{eq:faber_definition}
\end{equation}
In certain rare cases\red{,} the Chebyshev polynomials and Faber polynomials corresponding to a set coincide \cite{faber20}. \red{In other cases, the Faber polynomials provide near minimal upper bounds.} If $\sfE$ is the closure of an analytic Jordan domain \red{then}
\begin{equation}
    \|F_n^{\sfE}\|_{\sfE} = \Capacity(\sfE)^n\Big(1+O(r^n)\Big)
    \label{eq:faber_norms}
\end{equation}
for some $0<r<1$, see e.g. \cite{faber20,widom69}. This implies that the sequence $\{F_n^\sfE\}$ asymptotically saturates \eqref{eq:szego_inequality} \red{in this case and \[\frac{\|T_n^{\sfE}\|_{\sfE}}{\Capacity(\sfE)^n}
\rightarrow 1,\quad \text{as}\quad n\rightarrow \infty\] }\red{which is optimal as seen from \eqref{eq:szego_inequality}.} \red{The Widom factor was introduced in \cite{goncharov-hatinoglu15} and is defined as}
\begin{equation}
    \cW_n(\sfE)\coloneqq\frac{\|T_n^{\sfE}\|_{\sfE}}{\Capacity(\sfE)^n}.
    \label{eq:widom_factor}
\end{equation}
Much of the research into Chebyshev polynomials is directed to understanding the asymptotic behavior of $\cW_{n}(\sfE)$. In \cite{widom69} and later \cite{suetin74} conditions to guarantee that $\cW_n(\sfE)\rightarrow 1$ as $n\rightarrow \infty$ were relaxed by means of comparison with Faber polynomials. If $\sfE$ is the closure of a Jordan domain with $C^{1+\alpha}$ boundary then it follows from \cite[Theorem 2, p.68]{suetin84} that
\begin{equation}
	\cW_n(\sfE)\leq \frac{\|F_n\|}{\Capacity(\sfE)^n} = 1+O\left(\frac{\log n}{n^\alpha}\right).
\end{equation} It is an open question if the conditions concerning the regularity of the boundary can be further relaxed while still guaranteeing that the corresponding Widom factors converge to the theoretical minimal value. For instance, if $\sfE$ is the closure of a Jordan domain such that the bounding curve is piecewise-analytic but contains corner points can we still conclude that 
	\[\lim_{n\rightarrow \infty}\cW_n(\sfE)=1?\]
It is known that some level of smoothness of the bounding curve is required for $\cW_n(\sfE)\rightarrow 1$ to hold as there are known examples of fractal Jordan domains such that the Widom factors, at least along a subsequence, are bounded below by $1+\delta$ for some $\delta>0$. This can be deduced from results in \cite{kamo-borodin94}. 

\red{Through a comparison with Faber polynomials it can be shown that if $\sfE$ is a convex set then $\cW_n(\sfE)\leq 2$ for all $n$. A proof of this fact using an upper bound on the norm of the corresponding Faber polynomials deduced from \cite[Theorem 2]{kovari-pommerenke67} can be found in \cite[Theorem 10]{rubin24}. A strengthened form of the estimate is given in \cite[Theorem 4.1]{abdullayev-savchuk-tunc20}.}

\red{ Faber polynomials are not the only polynomials useful for studying Chebyshev polynomials. A completely different class of trial polynomials were used} in \cite{andrievskii-nazarov19} to prove that the sequence $\{\cW_n(\sfE)\}$ remains bounded if $\sfE$ is the closure of a quasi-disk. For examples illustrating the close interplay between Faber polynomials and Chebyshev polynomials, we refer the reader to \cite{saff-totik90, widom69}. In Section \ref{sec:computations} we will explore a possible relation between the Chebyshev and Faber polynomials that has been observed numerically. Loosely formulated this entails that Chebyshev polynomials approach Faber polynomials for a fixed degree along certain curves related to the conformal map of the set in question.
 
Besides understanding the norm behavior, another point of interest \red{lies in} understanding how the geometry of a set affects the zero distributions of the corresponding Chebyshev polynomials. Given a polynomial $P$, let $\nu(P)$ denote the normalized zero-counting measure of $P$. That is,
\[\nu(P) = \frac{1}{\deg(P)}\sum_{j=1}^{\deg(P)}\delta_{z_j}\red{,}\]
where $\delta_z$ is the Dirac delta measure at $z$ and $\{z_j\}$ denotes the zeros of $P$ counting multiplicity. Given a compact set $\sfE$, a typical quantitative way of describing the asymptotical distribution of the zeros of $T_n^{\sfE}$ is by determining weak-star limits of the sequence of measures $\{\nu(T_n^{\sfE})\}$. As it turns out, such weak-star limits are closely related to the potential theoretic concept of \red{the} equilibrium measure. We therefore introduce the notation $\mu_{\sfE}$ to denote the equilibrium measure corresponding to a compact set $\sfE$, see \cite[\S 3.3]{ransford95}. Given a sequence of degrees $\{n_k\}$\red{,} \cite[Theorem 2.1.7]{andrievskii-blatt01} says that if
\begin{equation}
    \lim_{k\rightarrow \infty}\nu(T_{n_k}^{\sfE})(M) = 0
    \label{eq:zeros_inside}
\end{equation}
for every compact set $M$ in the interior of $\sfE$ then $\nu(T_{n_k}^{\sfE})\xrightarrow{\ast}\mu_{\sfE}$ as $n_k\rightarrow \infty$. Loosely formulated, if ``almost all'' of the zeros of $T_{n_k}^{\sfE}$ approach the boundary then they distribute according to equilibrium measure. In particular, if $\sfE$ has empty interior then $\nu(T_{n}^{\sfE})\xrightarrow{\ast}\mu_{\sfE}$ as $n\rightarrow \infty$. It is shown in \cite{saff-totik90} that the zeros of $T_n^{\sfE}$\red{,} when $\sfE$ is the closure of a Jordan domain, stay away from the boundary precisely when the bounding curve is analytic. It therefore follows that if $\sfE$ is the closure of a Jordan domain whose boundary contains a corner\red{,} then the zeros of $T_n^{\sfE}$ will approach the boundary in some fashion. The question we want to investigate is if we can discern that \eqref{eq:zeros_inside} \red{is plausible} for such sets.

It could be argued that in order to study Chebyshev polynomials there are two available approaches. One \red{way} is to try to compare Chebyshev polynomials with other classes of polynomials which are candidates to provide small maximum norms such as the Faber polynomials. The other approach to studying Chebyshev polynomials -- which will be the main focus of this article -- is to consider computing these polynomials. In our case\red{,} these computations will be performed using numerical approximations. Such considerations are somewhat scarce in the literature\red{,} although examples exist which rely on other methods than the ones presented here. See for instance \red{\cite{foucart-lasserre19,grothkopf-opfer82, opfer76,thiran-detaille91, toh-trefethen98}}. In this article we will discuss and apply an algorithm suggested by P.~T.~P. Tang that was presented in his PhD thesis \cite{tang87} and further developed by B.~Fischer and J.~Modersitzki in \cite{fischer-modersitzki-93}. More specifically\red{,} we will compute Chebyshev polynomials corresponding to a wide variety of compact sets in the complex plane. Doing so, it will become apparent that certain hypothesis can be made plausible using numerical computations. See \cite{fischer-modersitzki-93, komodromos-russell-tang95, tang88} for further developments of this algorithm. 

\red{Another, well-known algorithm which could be used for the purpose of computing Chebyshev polynomials is the Lawson algorithm and its extensions, see \cite{lawson61}. One big difference between Tang's algorithm and the Lawson algorithm is that the Lawson algorithm computes the best approximation on a discrete subset. A detailed comparison between the performance of different algorithms that can be used to compute best approximations is conducted in \cite{tang87}.}

\red{For examples of studies of orthogonal polynomials based on experimental approaches, see \cite{alpan-goncharov-simsek18, kruger-simon15}.}

\subsection{Outline}
This article is organized as follows.

In Section \ref{sec:algorithm}, a short discussion concerning Tang's algorithm from \cite{tang87} is presented. In particular\red{, we demonstrate how it can be used to compute} Chebyshev polynomials. This section serves as the method part of the article. A pseudocode implementation is provided in the appendix as Algorithm \ref{alg:remez_algorithm}. 

In Section \ref{sec:computations}, we present numerical findings related to computations of Chebyshev polynomials using Tang's algorithm. In particular, Widom factors and zeros are computed for regular polygons, the $m$-cusped hypocycloid, circular lunes and the Bernoulli lemniscate. We also compare the difference between Chebyshev polynomials and Faber polynomials for \red{some of these} sets.

In Section \ref{sec:discussion}, the results from Section \ref{sec:computations} are discussed and we \red{formulate hypotheses about the behavior of Chebyshev polynomials} based on these \red{numerical experiments}. \red{Our main hypothesis is that the asymptotic behavior of Faber polynomials and Chebyshev polynomials is closely related}.

\red{\section*{Acknowledgement}
The authors are grateful to Frank Wikstr\"{o}m and Jacob Stordal Christiansen, both
at Lund University. The former for providing excellent advice on the presentation of the numerical experiments in this manuscript and the latter for helpful
discussions on the material presented herein. The anonymous referee is also thanked for constructive comments.
}

\red{The research of L. H\"{u}bner was supported by Odysseus grant number G0DDD23N from the FWO.}

\section{Numerical computations of Chebyshev polynomials}
\label{sec:algorithm}

In the following, we consider the procedure of approximating complex-valued functions on a \red{	connected} compact subset of the complex plane, henceforth denoted $\sfE$. \red{In line with the setting considered in \cite{tang87,tang88}, we restrict ourselves to real linear spaces} in the sense that all scalars appearing in linear combinations will be real-valued. Since any $k$-dimensional complex space can be regarded as a $2k$-dimensional space over the real numbers \red{this does not impose any restriction}. We introduce the notation $\cC_\R(\sfE)$ to denote the linear space of complex-valued continuous functions on $\sfE$, \red{where linear combinations are formed with real scalars}. We further let $V$ denote an $n$-dimensional subspace of $\cC_{\R}(\sfE)$ with an associated basis $\{\varphi_k\}_{k=1}^{n}$. The algorithm developed by Tang computes the best approximation $\varphi^\ast$ to $f$ \red{in the space} $V$. In other words,
\[ \|f-\varphi^\ast\|_\sfE\leq \|f-\varphi\|_\sfE\]
for every $\varphi\in V$. We assume throughout that $\varphi^\ast$ is unique. This will be the case when studying Chebyshev polynomials on a continuum, \red{that is,} a compact connected set containing infinitely many points. To conform to the case of Chebyshev polynomials, we let $f(z) =z^n$ and $\varphi$ denote a complex polynomial of degree at most $n-1$.

As usual, we let $\cC_\R(\sfE)^\ast$ denote the dual space of $\cC_\R(\sfE)$ and $V^\perp$ those linear functionals in $\cC_\R(\sfE)^\ast$ that vanish on $V$. The Riesz representation theorem states that any real linear functional in $\cC_\R(\sfE)^\ast$ can be represented through the formula
\[Lf = \mathrm{Re}\int_{\sfE}fd\mu,\]
where $\mu$ is a complex Borel measure. The extension theorem of Hahn--Banach implies an elementary relation between linear functionals and distance minimizing elements for Banach spaces. From \cite[Theorem 7 in \S 8.2]{lax02} we see that
\begin{equation}
    \min_{\varphi\in V}\|f-\varphi\|_{\sfE} = \max_{\substack{L\in V^\perp \\ \|L\|\leq 1}}|Lf|.
    \label{eq:duality_theorem}
\end{equation}
As stated, \eqref{eq:duality_theorem} provides no substantial information on the actual maximizing linear functional. The space of all complex Borel measures on $\sfE$ may prove too unwieldy to deal with in any practical situation. However, there exists maximizing linear functionals satisfying \eqref{eq:duality_theorem} with a specific simple form as was shown by Zuhovicki\u{\i} and Remez, see e.g. \cite[Theorem 2, p. 437]{smirnov-lebedev68}. The value in \eqref{eq:duality_theorem} coincides with the maximal value of all expressions of the form 
\begin{equation}
    L_{\boldr,\boldalpha,\boldz}(f) = \sum_{j=1}^{n+1}r_j\mathrm{Re}(e^{-i\alpha_j}f(z_j))
    \label{eq:maximizing_linear_functional_form}
\end{equation}
where $\boldr = \{r_j\}_{j=1}^{n+1}\in [0,1]^{n+1}$, $\boldalpha = \{\alpha_j\}_{j=1}^{n+1}\in [0,2\pi)^{n+1}$ and $\boldz=\{z_j\}_{j=1}^{n+1}\in \sfE^{n+1}$ are subject to the following constraints:
\begin{equation}
	\sum_{j=1}^{n+1}r_j = 1,
    \label{eq:condition_convex} 
\end{equation}
\begin{equation}
    L_{\boldr,\boldalpha,\boldz}(\varphi)=\sum_{j=1}^{n+1}r_j\mathrm{Re}(e^{-i\alpha_j}\varphi(z_j)) = 0,\quad \forall \varphi\in V.
    \label{eq:condition_annihilate}
\end{equation}
The goal of using Tang's algorithm, which is further illustrated in Appendix \ref{sec:appendix}, \red{is to compute the maximizing functional as in \eqref{eq:maximizing_linear_functional_form}}. The algorithm produces a sequence of linear functionals $\{L_{\boldr^{(\nu)},\boldalpha^{(\nu)},\boldz^{(\nu)}}\}$ together with an associated sequence of approximants $\{\varphi^{(\nu)}\}$, \red{such that} $L_{\boldr^{(\nu)},\boldalpha^{(\nu)},\boldz^{(\nu)}}(f)$ is increasing in $\nu$ and
\begin{equation}
	L_{\boldr^{(\nu)},\boldalpha^{(\nu)},\boldz^{(\nu)}}(f)\leq \|f-\varphi^\ast\|_{\sfE}\leq \|f-\varphi^{(\nu)}\|_{\sfE}.
	\label{eq:linear_functional_sup_norm_inequality}
\end{equation}
 
\red{Tang's algorithm will be applied until a wanted accuracy level, in terms of the relative error given by 
\begin{equation}
	\frac{\|f-\varphi^{(\nu)}\|_\sfE-L_{\boldr^{(\nu)},\boldalpha^{(\nu)},\boldz^{(\nu)}}(f)}{L_{\boldr^{(\nu)},\boldalpha^{(\nu)},\boldz^{(\nu)}}(f)}
\end{equation} is reached. Throughout, any reference to the precision of our computations will be in terms of this relative error.}
 
One of the \red{novel features of} Tang's algorithm in comparison to previous algorithms at the time of its inception is that it can be shown to converge quadratically if certain conditions are met, see \cite{tang87,tang88} for further details. \red{Assuming} that $\boldr^{(\nu)}>0$ for all sufficiently large $\nu\in \N$ then
\[\liminf_{\nu\rightarrow \infty}\frac{\|f-\varphi^{(\nu)}\|_\sfE-L_{\boldr^{(\nu)},\boldalpha^{(\nu)},\boldz^{(\nu)}}(f)}{L_{\boldr^{(\nu)},\boldalpha^{(\nu)},\boldz^{(\nu)}}(f)} = 0.\]
A simple proof of this can be found in \cite{fischer-modersitzki-93}. \red{We note that in our computations of Chebyshev polynomials, rapid convergence is typically observed.}
\section{Computations of Chebyshev polynomials}
\label{sec:computations}
We now turn to the computation of Chebyshev polynomials in the complex plane \red{and} stress the fact that this section will only contain computational results. \red{These results will be discussed} in Section \ref{sec:discussion}. \red{Let $\sfE$ be a polynomially convex set in the complex plane, by the Maximum principle $T_n^{\sfE} = T_n^{\partial \sfE}$.} To translate the notation from Section \ref{sec:algorithm}, we let $n$ be a specified degree and $\gamma:[0,1]\rightarrow \C$ a parametrization of \red{$\partial \sfE$}. In order to compute $T_n^{\sfE}$ we let $f(t) = \gamma(t)^n$ and in the general case, we choose the basis as
\begin{align*}
    \begin{bmatrix}
    \varphi_1(t)& \varphi_2(t) & \cdots & \varphi_n(t) & \varphi_{n+1}(t) & \varphi_{n+2}(t) & \cdots & \varphi_{2n}(t)
\end{bmatrix}\\ = \begin{bmatrix}
    1 & \gamma(t) & \cdots & \gamma(t)^{n-1} & i & i\gamma(t) & \cdots & \red{i\gamma(t)^{n-1}}
\end{bmatrix}.
\end{align*}
The algorithm, applied to this setting, \red{is designed to} produce coefficients $\lambda_1,\red{\dots,}\,  \lambda_{2n}$ such that
\[T_n^{\sfE}(z) = z^{n}-\sum_{k=1}^{n}(\lambda_k+i\lambda_{n+k})z^{k-1}.\]
In many cases it is possible to exploit the symmetry of a set to reduce the size of the basis which significantly helps with speeding up the computation. As an example\red{,} if $\sfE$ is \red{symmetric with respect to complex conjugation,} meaning that
\[z\in \sfE\quad\Leftrightarrow \quad\overline{z}\in \sfE,\]then by the uniqueness of $T_n^{\sfE}$ all coefficients must be real. Hence the basis can be chosen to be the $n$-dimensional real linear space spanned by
\[\varphi_k(t) = \gamma(t)^{k-1},\quad k = 1,\dotsc,n.\]
\red{Using the following lemma, it is possible to exploit the symmetry of the underlying set in order to make further reductions on the size of the basis used in Tang's algorithm}, see also \cite[Example 4.1]{christiansen-simon-zinchenko-V}.
\begin{lemma}
    Let $\sfE$ denote a compact infinite set satisfying
    \begin{equation}
    	\red{\sfE = \{ e^{2\pi i/m}z: z\in \sfE\}}
    	\label{eq:symmetry_set}
    \end{equation}
    \red{for some $m\in \N$}. \red{If} $n\in \N$ and $l\in \{0,1\dotsc,m-1\}$ \red{then}
    \begin{equation}
    	T_{nm+l}^{\sfE}(z) = z^{nm+l}+\sum_{k=0}^{n-1}a_kz^{km+l} = z^lQ_n(z^m)
    	\label{eq:symmetry_relation}
    \end{equation}
    where $Q_n$ denotes a monic polynomial of degree $n$.
    \label{lem:symmetry}
\end{lemma}
\begin{proof}
    The proof is an easy consequence of the uniqueness of Chebyshev polynomials. Considering the polynomial
    \[\red{e^{-2\pi il/m}T_{nm+l}^{\sfE}(e^{2\pi i/m}z) = z^{nm+l}+\text{lower order terms},}\]
    we see that this is a monic polynomial \red{and from \eqref{eq:symmetry_set} we conclude that it has} the same norm as $T_{nm+l}^{\sfE}$ on $\sfE$. From \red{the} uniqueness of the corresponding Chebyshev polynomial  
    \[\red{e^{-2\pi il/m}T_{nm+l}^{\sfE}(e^{2\pi i/m}z) = T_{nm+l}^{\sfE}(z)}.\]
\red{The coefficient in front of $z^{j}$ on the left-hand side is equal to $e^{2\pi i(j-l)/m}$ multiplied with the coefficient in front of $z^j$ on the right-hand side. Hence, if $j-l$ fails to be a multiple of $m$ then this coefficient must be $0.$ This establishes that \eqref{eq:symmetry_relation} holds.}
\end{proof}

We will consider the computation of Chebyshev polynomials corresponding to a variety of sets \red{for which their asymptotic behavior remains} unknown. First\red{, we compute the} Widom factors $\cW_{n}$. Next\red{, we explore} a possible connection between Chebyshev polynomials and Faber polynomials using numerical experiments. Finally\red{, we examine the distribution of zeros of} $T_n^{\sfE}$.

\red{All computations are carried out in Python using the arbitrary-precision arithmetic module \texttt{mpmath}. For each computation, we use a working accuracy of at least 50 significant digits to avoid phenomena caused by numerical instabilities.}

%\begin{remark}
%	Let us heavily emphasize the fact that the computations performed here will provide \red{an} $n$th degree monic polynomials $P_n$ such that $\|P_n\|_{\sfE}$ is close to the theoretical minimum $\|T_n^{\sfE}\|_{\sfE}$. Furthermore, $\|P_n\|_{\sfE}-\|T_n^{\sfE}\|_{\sfE}$ can be explicitly upper bounded in the computations using \eqref{eq:linear_functional_sup_norm_inequality}. This implies that Widom factors can be accurately estimated. Regarding intricate polynomial properties such as their coefficients and zeros, the algorithm has to be used with care. Although it is true that if $\|P\|_{\sfE}$ is close to $\|T_n^{\sfE}\|_{\sfE}$ then their distance is small in every measurable way, it is in general difficult to quantify this. We remark however, that the computations are consistent in the sense that the behaviors here exhibited do not change as the precision is increased further.
%\end{remark}

\subsection{Computations of Widom factors}
As was already stated in Section \ref{sec:intro}, recall that if $\sfE$ denotes the closure of a Jordan domain with $C^{1+\alpha}$ boundary then it is known that $\cW_{n}(\sfE)\rightarrow 1$ as $n\rightarrow \infty$, see \cite{suetin74, suetin84,widom69}. If $\sfE$ is convex it is possible to conclude that $\cW_n(\sfE)\leq 2$, see \cite{rubin24}. If $\sfE$ is a quasi-disk then $\cW_n(\sfE)$ is known to be bounded \cite{andrievskii-nazarov19}. Likewise, \red{if the} outer boundary of $\sfE$ consists of \red{Dini}-smooth arcs which are disjoint apart from their endpoints \red{and do not form inward pointing cusps} \red{then} the sequence \red{$\{\cW_n(\sfE)\}$} is bounded, see \cite[Theorem 2.1]{totik-varga15}. Informally stated, an \red{inward pointing cusp (also external cusp in the literature)} is a point where the intersecting arcs form an angle of $2\pi$ \red{within} the interior of $\sfE$ so that \red{it} ``points away'' from the unbounded complement \red{of $\sfE$}. Apart from these results, very few general estimates exist regarding Widom factors for compact sets, even with the additional assumption that they are closures of Jordan domains.

Note that $\cW_n(\sfE)$ is invariant under dilations and translations in the sense that for any $\alpha,\beta\in \C$ with $\alpha\neq0$ we have
\[
\cW_n(\alpha \sfE +\beta) = \cW_n(\sfE),
\]
see \cite{goncharov-hatinoglu15}. \red{Thus, the set may always be rotated and scaled so that} symmetries can be easily exploited without affecting the Widom factors. \red{The results of the numerical computations are presented below, while their discussion is deferred to} Section \ref{sec:discussion}.

\subsubsection{Regular \red{polygons}}
Simple examples of piecewise-analytic Jordan domains with corners are the regular polygons (also called $m$-gons) if they have $m$ sides of equal length. Due to the convexity of such sets we immediately gather that if $\sfE$ is a regular polygon then $\cW_n(\sfE)\leq 2$. It is not known whether the sequence $\{\cW_n(\sfE)\}$ converges in this case and \red{we thus study} the corresponding Widom factors numerically. Previous numerical \red{studies} for Chebyshev polynomials corresponding to a square have been \red{conducted} in \cite{opfer76} for degrees up to 16. These, however, lack the perspective of Widom factors. The logarithmic capacity of a regular $m$-gon $\sfE$ can be found in \cite[Table 5.1]{ransford95}. It is stated there that
\begin{equation}
    \Capacity(\sfE) = \frac{\Gamma(1/m)}{2^{1+2/m}\pi^{1/2}\Gamma(1/2+1/m)}\cdot \text{side length}(\sfE).
    \label{eq:ngon_capacity}
\end{equation}
We use this formula together with Tang's algorithm to compute the Widom factors corresponding to different $m$-gons. If the corners are located at
\[\left\{\exp\left(\frac{2\pi i k}{m}\right)\Big\vert~ k = 0,1,\dotsc,m-1\right\},\]
then the set is invariant under rotations by an angle of $2\pi /m$ and hence Lemma \ref{lem:symmetry} implies that
\begin{equation}
	T_{nm+l}^{\sfE}(z) = z^lQ_n^{\sfE}(z^m), \quad l = 0,1,\dotsc,m-1,
	\label{eq:mgon_representation}
\end{equation}
where $Q_n^{\sfE}$ is a monic polynomial of degree $n$, depending on $m$, whose coefficients are all real. From \eqref{eq:mgon_representation} it follows that $n$ basis elements are needed in Tang's algorithm to compute $T_{nm+l}^{\sfE}$. We use the following notation:
\begin{itemize}
    \item $\sfE_{\Delta}$ - the equilateral triangle, $m=3$,
    \item $\sfE_{\square}$ - the square, $m=4$,
    \item $\sfE_{\pentagon}$ - the pentagon, $m = 5$,
    \item $\sfE_{\hexagon}$ - the hexagon, $m = 6$.
\end{itemize}
The corresponding Widom factors are illustrated in Table \ref{tab:mgons} and Figures \ref{fig-triangle}--\ref{fig-hexagon} and will be further discussed in Section \ref{sec:discussion}. 
\begin{table}[h]
    \centering
    \begin{tabular}{l c c c c} 
        & $\cW_{n}(\sfE_\Delta)$ & $\cW_{n}(\sfE_\square)$& $\cW_{n}(\sfE_{\pentagon})$ & $\cW_{n}(\sfE_{\hexagon})$ \\ [0.5ex] 
		\hline        
$n = 5$&\red{1.30901051}&\red{1.27841716}&\red{1.21350890}&\red{1.51420435}\\
$n = 10$&\red{1.14268975}&\red{1.12981144}&\red{1.14236706}&\red{1.17363458}\\
$n = 25$&\red{1.05488942}&\red{1.04969579}&\red{1.05544736}&\red{1.06322465}\\
$n = 50$&\red{1.02708221}&\red{1.02449420}&\red{1.02724022}&\red{1.03142381}\\
$n = 90$&\red{1.01495704}&\red{1.01352749}&\red{1.01502916}&\red{1.01733310}\\
$n = 120$&\red{1.01119706}&\red{1.01012748}&\red{1.01124879}&\red{1.01297657}\\
    \end{tabular}
    \vspace{0.5cm}
    \caption{Widom factors corresponding to regular polygons, computed with an accuracy of \red{$10^{-10}$} using Tang's algorithm.}
    \label{tab:mgons}
\end{table}

\begin{figure}[h]
\centering
	\begin{subfigure}{.49\textwidth}
		\includegraphics[width=\linewidth]{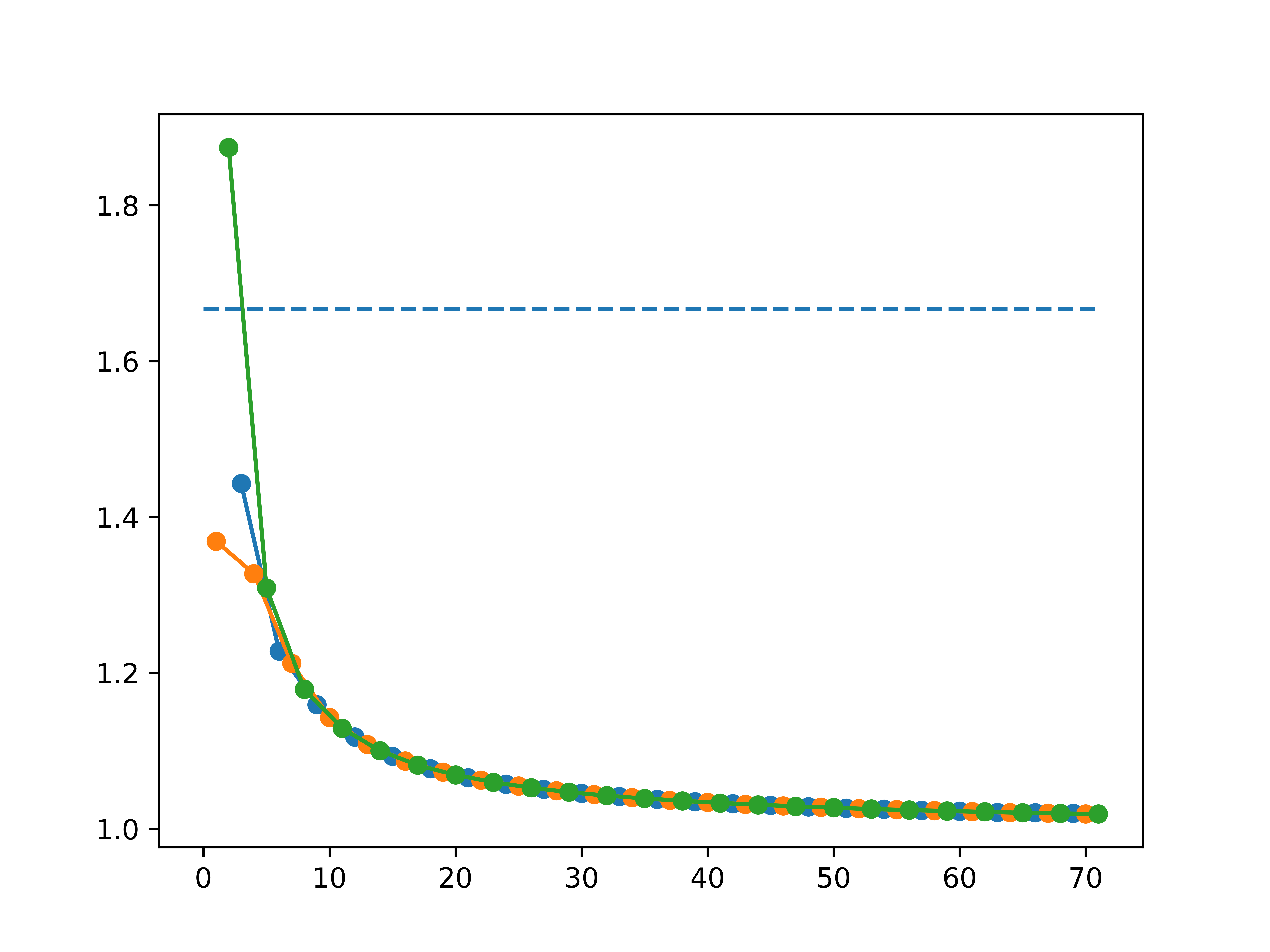}
	    \caption{\red{Widom factors of an equilateral triangle.}}
	    \label{fig-triangle}
	\end{subfigure}
	\begin{subfigure}{.49\textwidth}
		\includegraphics[width=\linewidth]{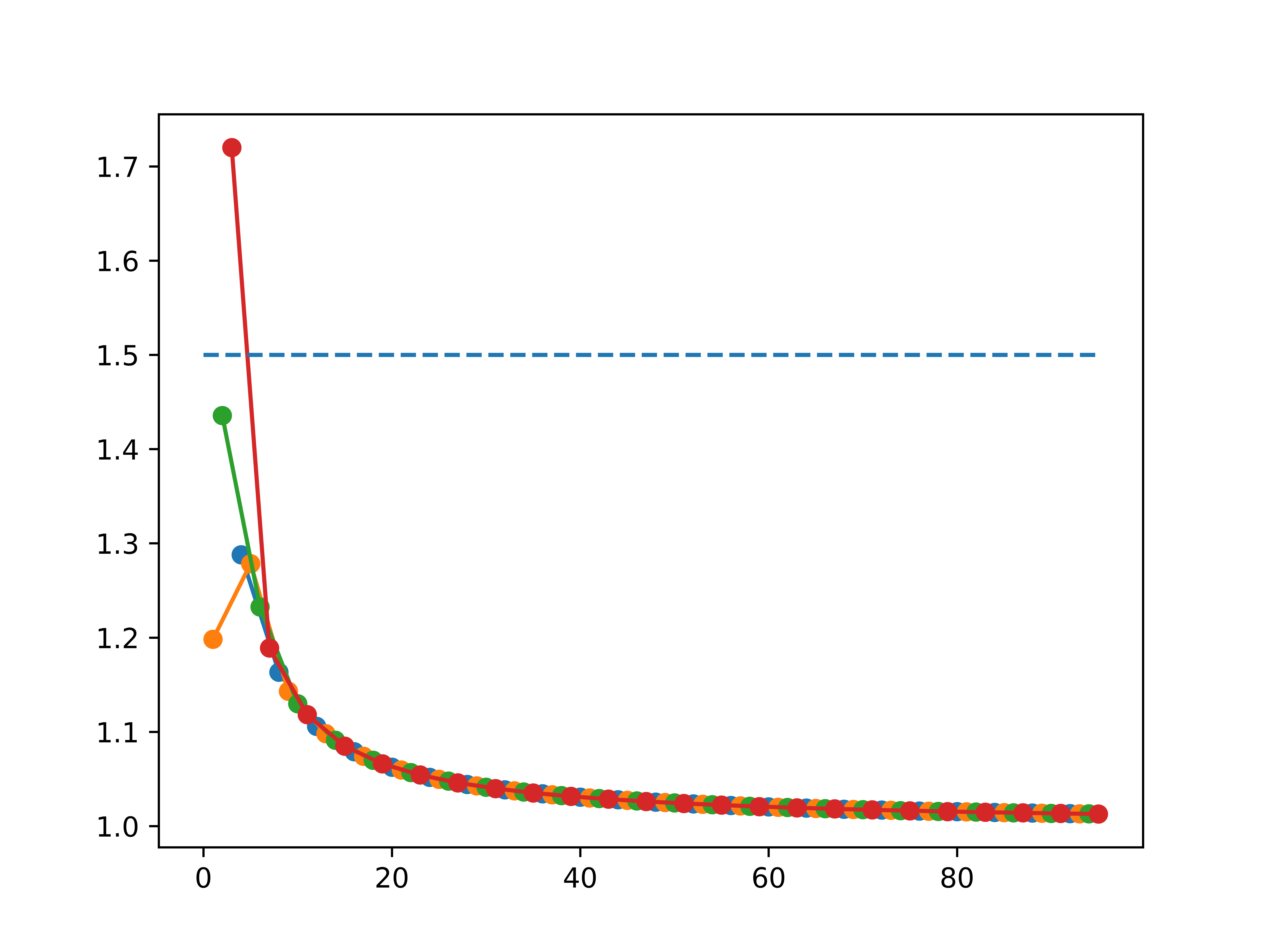}
	    \caption{\red{Widom factors of a square.}}
	    \label{fig-square}
	\end{subfigure}
	\begin{subfigure}{.49\textwidth}
		\includegraphics[width=\linewidth]{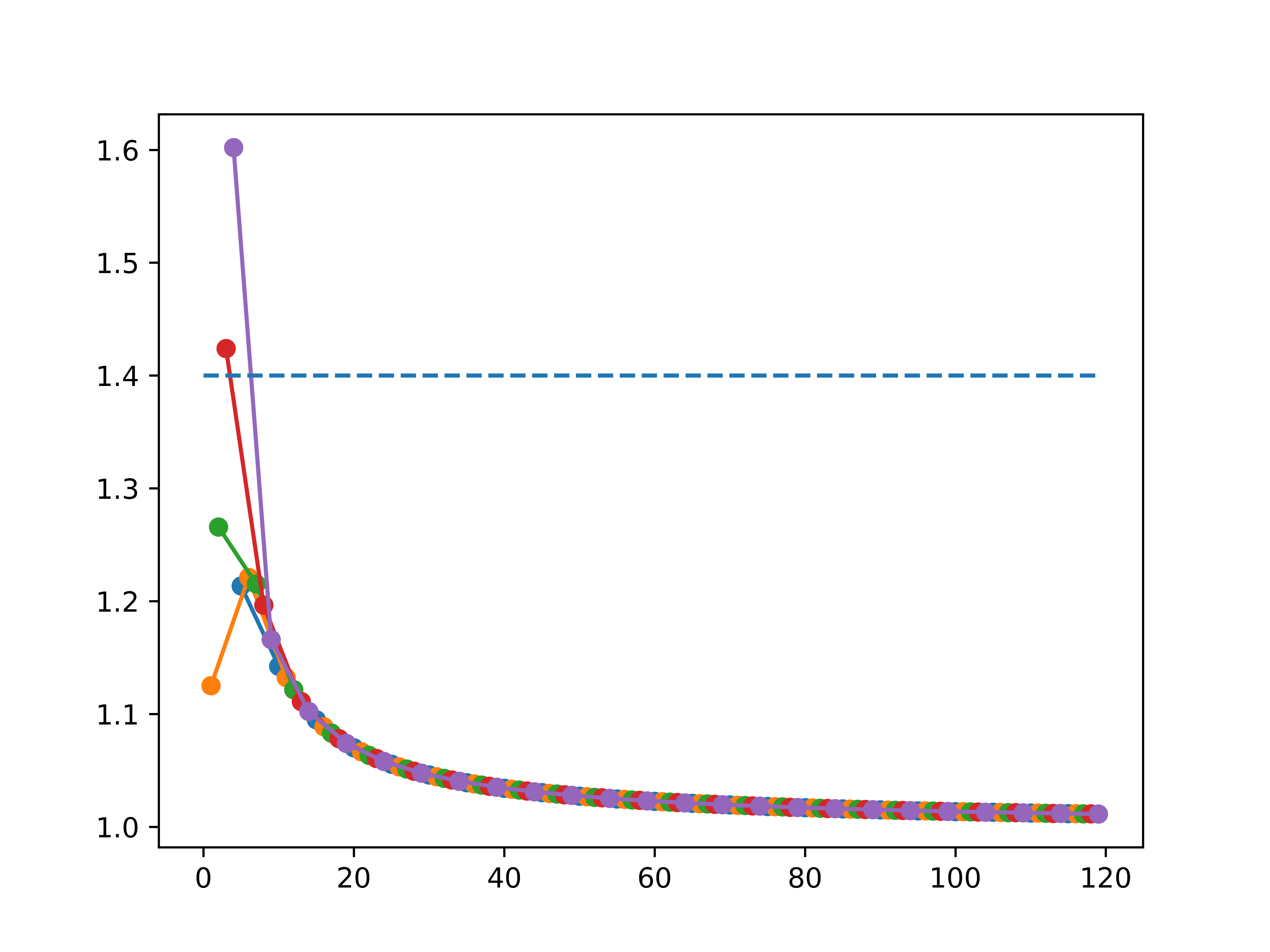}
	    \caption{\red{Widom factors of a regular pentagon.}}
	    \label{fig-pentagon}
	\end{subfigure}
	\begin{subfigure}{.49\textwidth}
		\includegraphics[width=\linewidth]{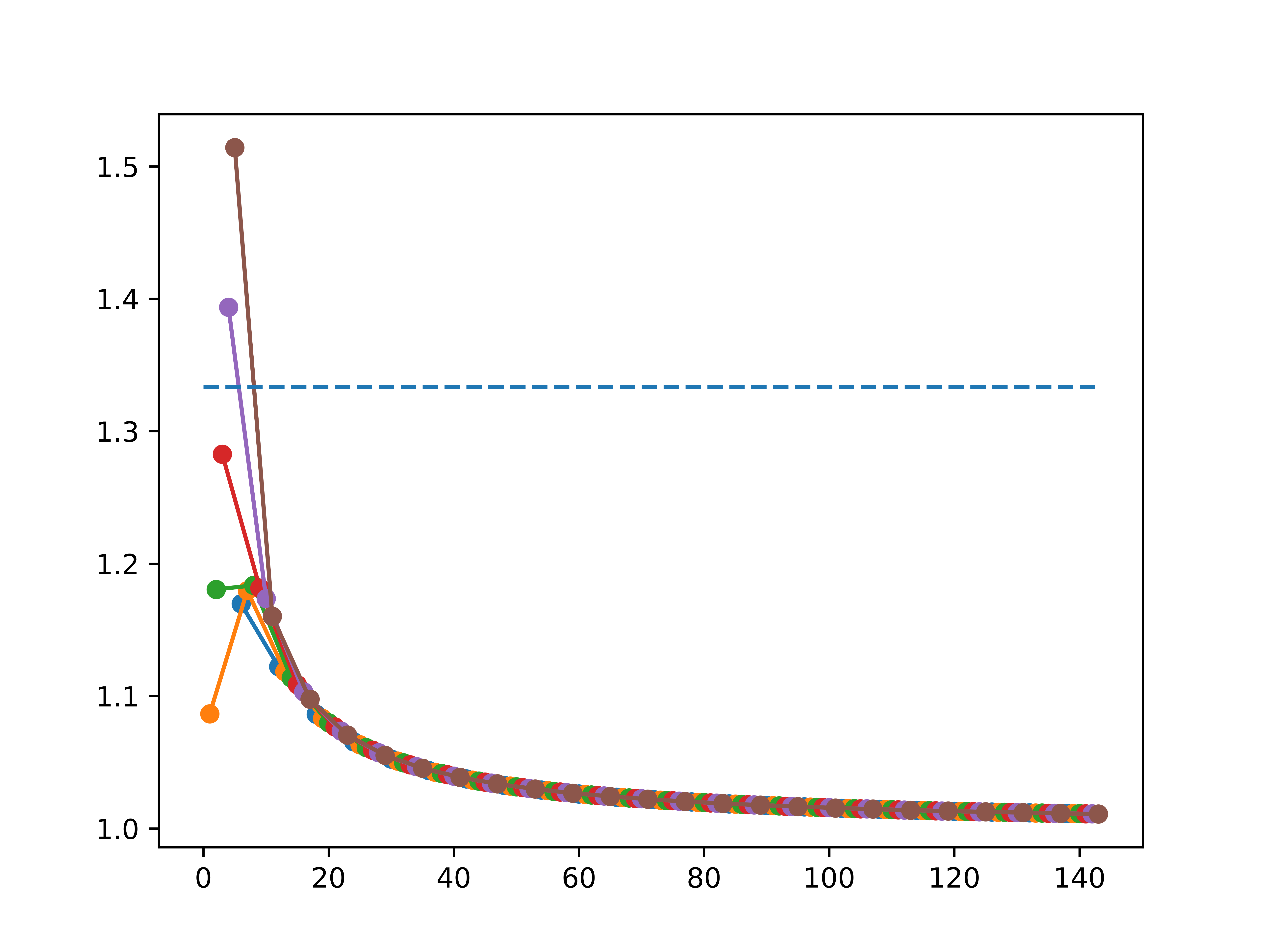}
	    \caption{\red{Widom factors of a regular hexagon.}}
    	\label{fig-hexagon}
	\end{subfigure}
	\caption{\red{The difference between $\cW_{n}$ and $1$ for different $m$-gons.} The dotted blue line represents the value \red{$\frac{m+2}{m}-1$} which relates to the Faber polynomials. \red{The Tang accuracy is $10^{-7}$.}}
\end{figure}
\subsubsection{Hypocycloids}
Examples of sets which are not quasi-circles are sets which contain cusps on their boundary. With an ``outward pointing cusp'' we simply mean that the exterior angle at such a point is $2\pi$. For a pictorial representation the reader can consult Figures \ref{fig-zeros-del}-\ref{fig-zeros-exo} \red{illustrating} $m$-cusped hypocycloids. These are the Jordan curves $\{\sfH_m\}$ defined via
\begin{equation}
    \sfH_m\coloneqq\left\{e^{i\theta}+\frac{e^{-i(m-1)\theta}}{m-1}: \theta\in [0,2\pi)\right\}.
    \label{eq:hypocycloid}
\end{equation}
It is easily seen that if $\Phi$ is the external conformal map from the unbounded component of $\C\setminus \sfH_m$ to $\{z:|z|>1\}$ satisfying $\Phi(z) = \Capacity(\sfH_m)^{-1}z+O(1)$ as $z\rightarrow \infty$ then
\[\Phi^{-1}(z) = z+\frac{z^{-(m-1)}}{m-1}\]
and hence $\Capacity(\sfH_m) = 1$ for any $m$.

The corresponding Faber polynomials have been studied in \cite{he96,he-saff94}. Particular focus has been directed toward the corresponding zero distributions which are confined to straight lines. 

The sets $\sfH_m$ are clearly invariant under rotations by $e^{2\pi i/m}$ and therefore Lemma \ref{lem:symmetry} implies that
\[T_{nm+l}^{\sfH_{m}}(z) = z^lQ_n^{\sfH_m}(z^m),\]
where $Q_n^{\sfH_m}$ is a monic polynomial with real coefficients. The corresponding Widom factors are illustrated in Table \ref{tab:hypocycloid} and Figures \ref{fig-hyp-3}--\ref{fig-hyp-6} and will be further discussed in \red{Section} \ref{sec:discussion}. 

\begin{table}[h]
    \centering
    \begin{tabular}{l c c c c} 
        & $\cW_{n}(\sfH_{3})$ & $\cW_{n}(\sfH_{4})$& $\cW_{n}(\sfH_{5})$ & $\cW_{n}(\sfH_{6})$ \\ [0.5ex] 
        \hline
$n = 5$&\red{1.69594045}&\red{1.52124467}&\red{1.64453125}&\red{2.48832000}\\
$n = 10$&\red{1.43149779}&\red{1.40782752}&\red{1.40910966}&\red{1.47443526}\\
$n = 25$&\red{1.25181361}&\red{1.24037027}&\red{1.24939099}&\red{1.26640101}\\
$n = 50$&\red{1.17181774}&\red{1.16257852}&\red{1.16744376}&\red{1.17588712}\\
$n = 90$&\red{1.12543411}&\red{1.11836230}&\red{1.12140197}&\red{1.12692589}\\
\red{$n = 120$}&\red{1.10776251}&\red{1.10161806}&\red{1.10410341}&\red{1.10869246}\\
        
    \end{tabular}
    \vspace{0.5cm}
    \caption{Widom factors corresponding to $m$-cusped Hypocycloids, computed with an accuracy of \red{$10^{-10}$} using Tang's algorithm.}
    \label{tab:hypocycloid}
\end{table}
\begin{figure}[h]
\centering
\begin{subfigure}{.49\textwidth}
    \centering
	    \includegraphics[width=\linewidth]{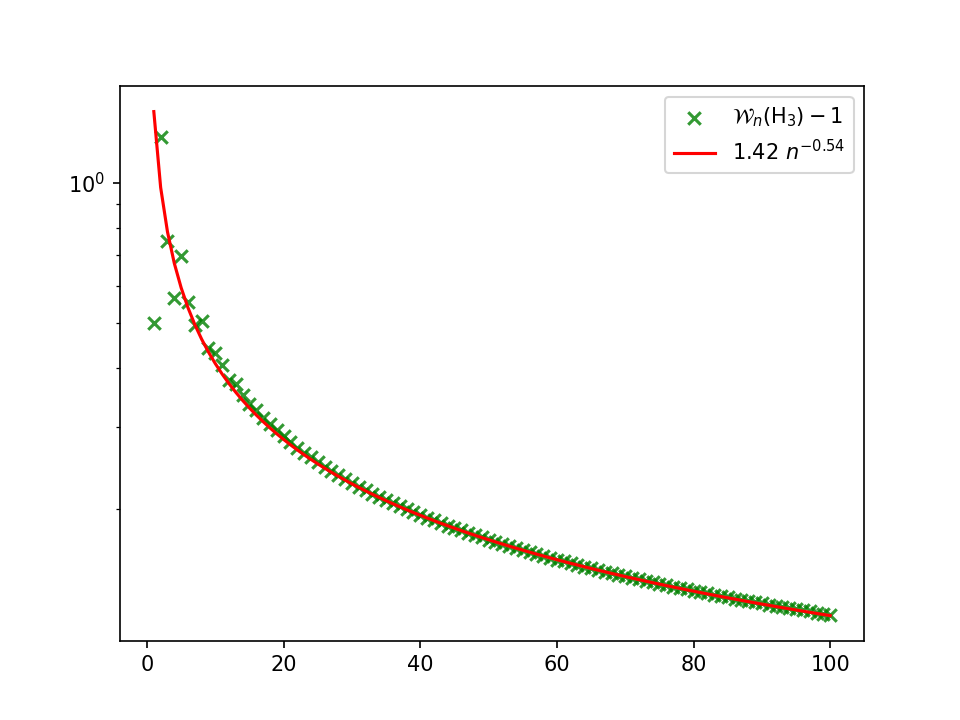}
    \caption{\red{Widom factors of $\sfH_3$.}}
    \label{fig-hyp-3}
\end{subfigure}%
%\ContinuedFloat
\begin{subfigure}{.49\textwidth}
    \centering
    \includegraphics[width=\linewidth]{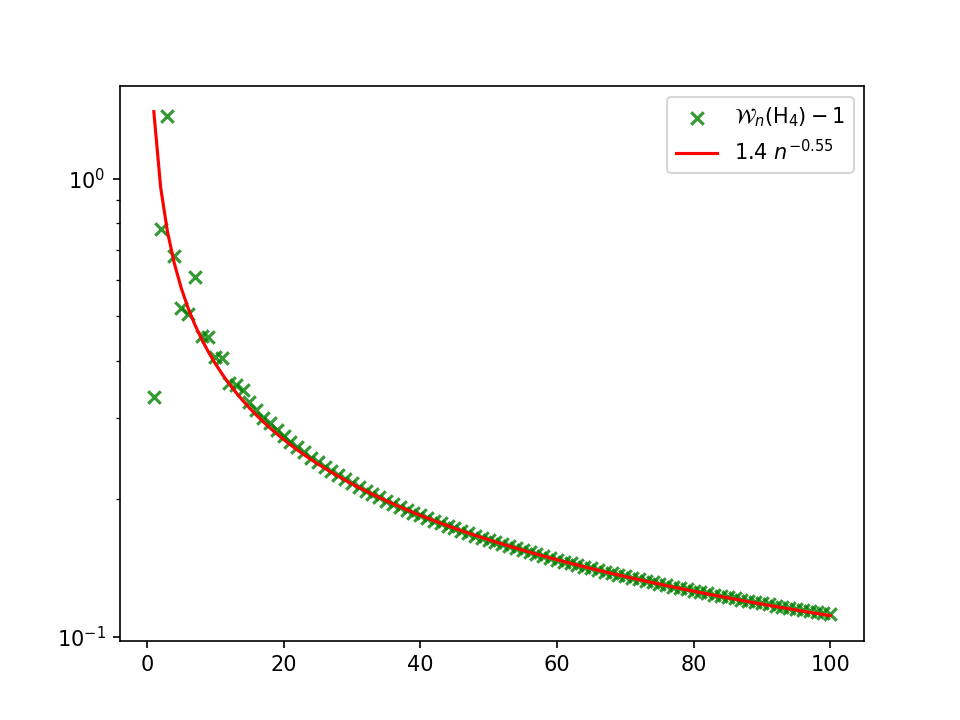}
    \caption{\red{Widom factors of $\sfH_4$.}}
    \label{fig-hyp-4}
\end{subfigure}
%\ContinuedFloat
\begin{subfigure}{.49\textwidth}
    \centering
    \includegraphics[width=\linewidth]{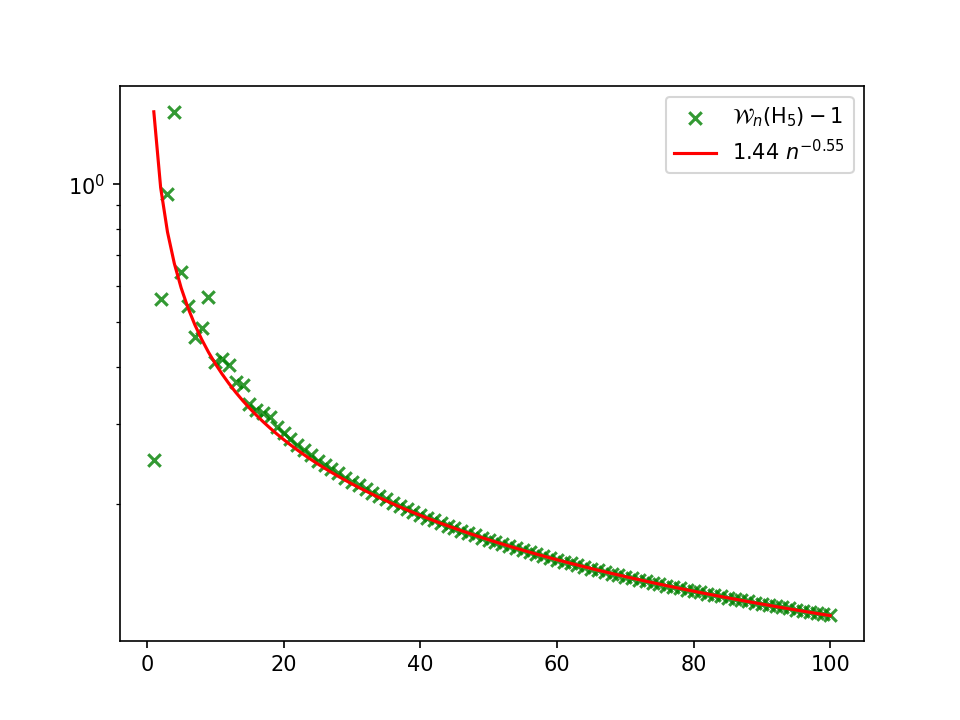}
    \caption{\red{Widom factors of $\sfH_5$.}}
    \label{fig-hyp-5}
\end{subfigure}%
\begin{subfigure}{.49\textwidth}
    \centering
    \includegraphics[width=\linewidth]{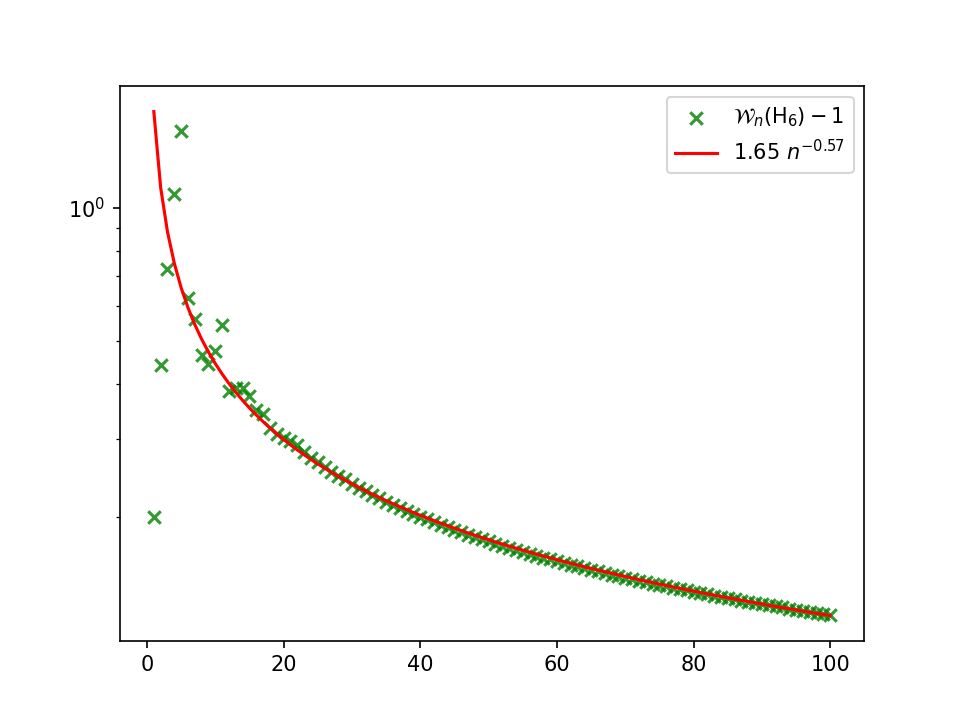}
    \caption{\red{Widom factors of $\sfH_6$.}}
    \label{fig-hyp-6}
\end{subfigure}%
\caption{\red{The difference between $\cW_n$ and $1$ on different $m$-cusped Hypocycloids. The Tang accuracy is $10^{-7}.$}}
\end{figure}	

\subsubsection{Circular lunes}
As a final example of the computation of Widom factors we consider the case of circular lunes, see Figures \ref{fig-zeros-lune-1-2} and \ref{fig-zeros-lune-3-2}. Given $\alpha \in (0,2]$, we let
\begin{equation}
    \label{eq:lune_def}
    \sfC_\alpha \red{\coloneqq} \left\{\alpha\frac{1+\left(\frac{w-1}{w+1}\right)^\alpha}{1-\left(\frac{w-1}{w+1}\right)^\alpha}: |w| = 1\right\}
\end{equation}
with vertices at $\pm\alpha$ and exterior angle $\alpha\pi$. \red{Note that $\Capacity(\sfC_\alpha) =1$.} The structure of such sets \red{depends} on the value of the parameter $\alpha$. \red{$\sfC_\alpha$ is non-convex if $\alpha \in (0,1)$ and convex if $\alpha \in (1,2)$}. The extreme cases are $\sfC_1 = \T$ and $\sfC_2 = [-2,2]$. We will consider two parameter values, namely $\alpha = \frac{1}{2}$ and $\alpha = \frac{3}{2}$ as they \red{have different convexity properties}. \red{Regardless} of the parameter value of $\alpha$, the set $\sfC_\alpha$ is symmetric with respect to both axes. From Lemma \ref{lem:symmetry} we conclude that
\[
T_{2n+l}^{\sfC_\alpha} = z^lQ_n^{\alpha}(z^2)
\]
where $Q_n^{\alpha}$ is a monic polynomial of degree $n$ with real coefficients. The results of the computations using Tang's algorithm are illustrated in Table \ref{tab:circular-lune} and Figures \ref{fig-lune-1-2} and \ref{fig-lune-3-2}.

\begin{table}[h]
    \centering
    \begin{tabular}{l c c} 

        & $\cW_{n}(\sfC_{1/2})$ & $\cW_{n}(\sfC_{3/2})$ \\ [0.5ex] 
		\hline
\red{$n = 5$}&\red{1.10286958}&\red{1.12569879}\\
$n = 10$&\red{1.03696888}&\red{1.06185388}\\
$n = 25$&\red{1.03405451}&\red{1.02444481}\\
\red{$n = 50$}&\red{1.01442556}&\red{1.01215983}\\
\red{$n = 90$}&\red{1.00936347}&\red{1.00673877}\\
\red{$n = 120$}&\red{1.00749065}&\red{1.00505004}
        
    \end{tabular}
    \vspace{0.5cm}
    \caption{Widom factors corresponding to circular lunes, computed with an accuracy of \red{$10^{-10}$} using Tang's algorithm.}
    \label{tab:circular-lune}
\end{table}

\begin{figure}[h]
\centering
\begin{subfigure}{.45\textwidth}
    \centering
    \includegraphics[width=\linewidth]{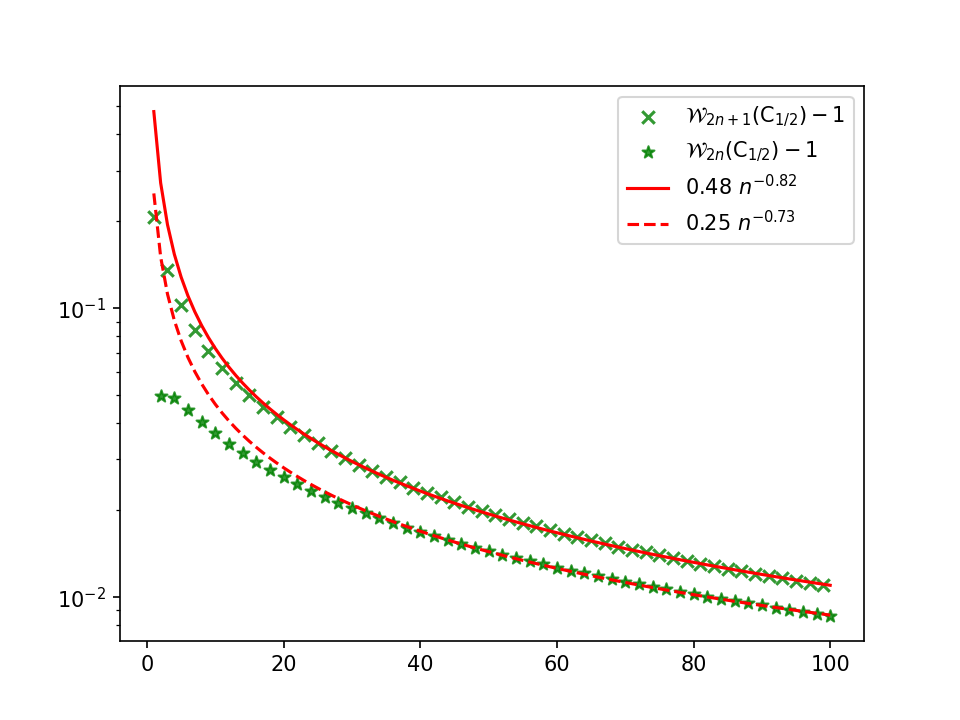}
    \caption{\red{Widom factors of $\sfC_{1/2}$.}}
    \label{fig-lune-1-2}
\end{subfigure}%
%\ContinuedFloat
\begin{subfigure}{.45\textwidth}
    \centering
    \includegraphics[width=\linewidth]{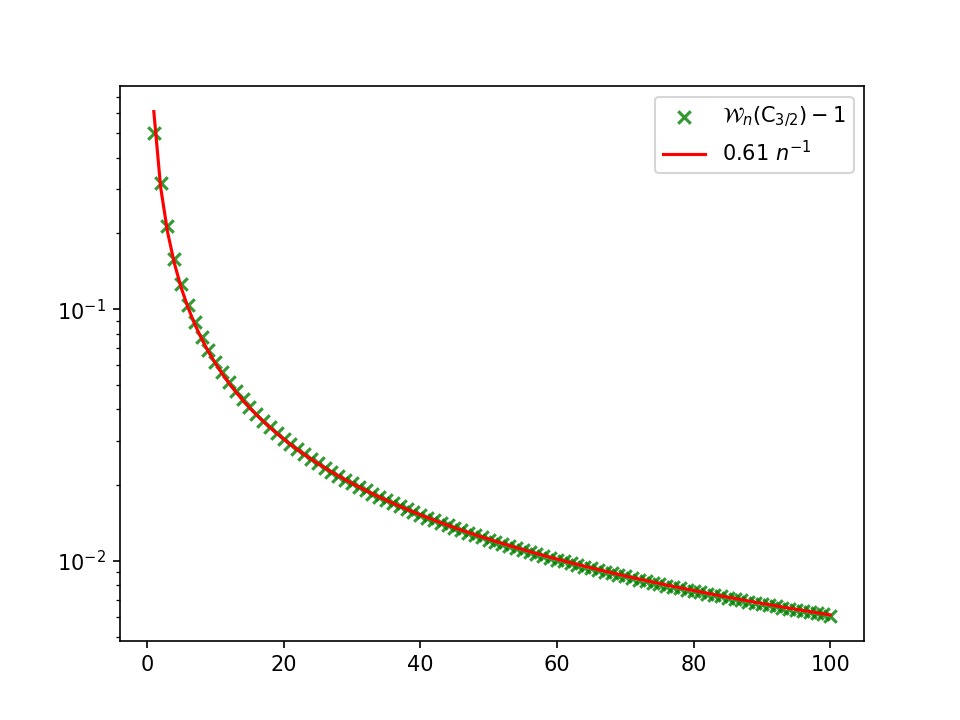}
    \caption{\red{Widom factors of $\sfC_{3/2}$.}}
    \label{fig-lune-3-2}
\end{subfigure}
\caption{\red{The difference betweeen $\cW_n$ and $1$ on two different circular lunes. The Tang accuracy is $10^{-7}.$}}
\end{figure}

\subsection{The Faber connection}
Our initial \red{motivation} for computing Chebyshev polynomials originated in \red{obtaining a better understanding of} the \red{behavior of their} zeros. \red{One case that we wished to better understand concerned} Chebyshev polynomials on \red{\textit{equipotential}} curves. More precisely, if $\Phi$ is the exterior conformal map of a connected set $\sfE$, we investigated Chebyshev polynomials on the level curves
\[\sfE^r \coloneqq \{z:|\Phi(z)| = r\}\]
and found that the corresponding zeros of $T_n^{\sfE^r}$ seemed to converge \red{as $r$ increases}. By simultaneously plotting the zeros of the Faber polynomials, the picture became quite clear. The zeros of $T_n^{\sfE^r}$, as $r$ increased, appeared to \red{converge toward} the zeros of the corresponding Faber polynomials. \red{We now investigate} this possible relation numerically for lemniscates, hypocycloids and circular lunes.

\subsubsection{Lemniscates}
For given parameters $r>0$ and $m\in \N$, we define a family of compact lemniscatic sets via
\begin{equation}
    \sfL_{m}^r \coloneqq \Big\{z:|z^m-1| = r^m\Big\}.
\end{equation}
A pictorial representation of such sets can be found in Figures \ref{fig-zeros-bernoulli} and \ref{fig-zeros-peanut}. From \eqref{eq:capacity_preimage} we \red{observe} that $\Capacity(\sfL_{m}^{r}) = r$ and since the polynomial $(z^m-1)^n$ saturates the lower bound in \eqref{eq:szego_inequality} we see that $T_{nm}^{\sfL_{m}^{r}}(z) = (z^m-1)^n$. For the remaining degrees we apply Lemma \ref{lem:symmetry} to \red{conclude}
\begin{equation}
    T_{nm+l}^{\sfL_m^r}(z) = z^{l}Q_{n}^{\sfL_m^r}(z^m),
    \label{eq:cheb_lemniscate}
\end{equation}
where $Q_n^{\sfL_m^r}$ is a monic polynomial whose coefficients are all real. The parameter $r$ determines three separate regimes of sets.
\begin{itemize}
    \item If $r>1$ then $\sfL_m^r$ is the closure of an analytic Jordan domain.
    \item If $r = 1$, we write $\sfL_m^{1} = \sfL_m$ and in this case $\sfL_m$ is connected, but its interior is not.
    \item If $0<r<1$ then $\sfL_m^{r}$ consists of $m$ components.
\end{itemize}
Since $T_{mm}^{\sfL_m^r}(z) = (z^m-1)^n$, we see that $\cW_{nm}(\sfL_m^{r}) = 1$ for any $n, m$ and $r$. The question is what the asymptotic behavior is for the remaining sequences of degrees. For $r>1$ it follows immediately from \eqref{eq:faber_norms} that $\cW_{n}(\sfL_m^r)\rightarrow 1$\red{,} as $n\rightarrow \infty$ since the boundary is an analytic Jordan curve. If $0<r<1$ then it is known that $\limsup_{n\rightarrow \infty}\cW_{n}(\sfL_m^r)>1$, see \red{\cite[Theorem 2]{totik14}}. The remaining case, when $r = 1$, is handled by \cite[Corollary 2]{bergman-rubin24} where it is shown that $\cW_{n}(\sfL_m^1)\rightarrow 1$ as $n\rightarrow \infty$. 

In the following discussion we limit ourselves to the case $m=2$ and write $\sfL^r=\sfL_2^r$ and $\sfL = \sfL^1$. It should be stressed that analogous considerations are possible for any $m$. The set $\sfL$ is the classical Bernoulli lemniscate.

The conformal map\red{ping from} $\overline{\C}\setminus \sfL$ to $\{z:|z|>1\}$ with $\Phi(\infty) = \infty$ is given by 
\[\Phi(z) = \sqrt{z^2-1},\]
where the branch is chosen such that $\Phi(z) =z+O(1)$ at infinity. It follows from \eqref{eq:faber_definition} that
\[F_{2n}^{\sfL}(z) = (z^2-1)^n\]
and hence $T_{2n}^{\sfL} = F_{2n}^{\sfL}$ for any value of $n$. We investigate if there is a possible relation between $F_{2n+1}^{\sfL}$ and $T_{2n+1}^{\sfL^r}$ as well.

It is possible to determine the Chebyshev polynomial of degree $3$ corresponding to $\sfL^r$ explicitly by solving the system of equations
\[\begin{cases}
    \frac{\partial }{\partial \theta}|z(z^2+a)|^2 = 0 \\
    \frac{\partial}{\partial a}|z(z^2+a)|^2 = 0
\end{cases}\]
with $z = \sqrt{r^2e^{i\theta}+1}$. For $r\geq 1$ a computation shows that the solution is given by
\begin{equation}
        T_{3}^{\sfL^r}(z) = z\left(z^2-\frac{1}{5}\left(4-r^4+\sqrt{1+7r^4+r^8}\right)\right).
        \label{eq:cheb_lemniscate_deg_three}
\end{equation}
On the other hand, using the Taylor expansion of $\Phi$ it is easy to see that
\[F_3^{\sfL}(z) = z\left(z^2-\frac{3}{2}\right)\]
and hence we gather from \eqref{eq:cheb_lemniscate_deg_three} that \red{$T_{3}^{\sfL^r}(z) \rightarrow  F_{3}^{\sfL}(z)$ as $r\rightarrow \infty$}, uniformly on compact subsets of the complex plane. The question is whether this should be considered an anomaly or a potential link between Chebyshev polynomials and Faber polynomials. The natural procedure is of course to consider further examples. We do so numerically using Tang's algorithm.

We define a norm on polynomials in the following way. If $P(z) = \sum_{k=0}^{n}a_kz^k$ then $\|\cdot\|_\infty$ is given by
\begin{equation}
    \|P\|_\infty= \max_{0\leq k\leq n}|a_k|.
    \label{eq:polynomial_norm}
\end{equation}
Our aim with this is to display the difference 
\[\|T_{2n+1}^{\sfL^r}-F_{2n+1}^{\sfL}\|_{\infty}\]
and illustrate that this appears to tend to $0$ with \red{$r\rightarrow \infty$}, see Figure \ref{fig-faber}.

\subsubsection{Hypocycloids}
We continue the considerations concerning a possible relation between Faber polynomials and Chebyshev polynomials on level curves corresponding to conformal maps. We therefore return to the family of $m$-cusped hypocycloids $\{\sfH_m\}$. The Faber polynomials $F_{n}^{\sfH_m}$ can be computed using \cite[Proposition 2.3]{he-saff94}.

For $r>1$, we let
\[\sfH_{m}^r\coloneqq\left\{re^{i\theta}+\frac{(re^{i\theta})^{-(m-1)}}{m-1}: \theta\in [0,2\pi)\right\}.\]
If $\Phi$ denotes the external conformal map from the unbounded component of $\C\setminus \sfH_m$ to $\{z:|z|>1\}$ with $\Phi(z) = z+O(1)$ as $z\rightarrow \infty$ then $\sfH_{m}^{r}$ is the analytic Jordan curve where $\Phi$ attains modulus $r$. With the intention of considering the possibility that 
\[T_{n}^{\sfH_m^r}\rightarrow F_n^{\sfH_m}\]
as $r\rightarrow \infty$, we compute $\|F_n^{\sfH_m}-T_{n}^{\sfH_m^r}\|_\infty$ for $m = 5$, and illustrate it in Figure \ref{fig-faber}.

\subsubsection{Circular lunes}
We end the considerations of comparing Chebyshev polynomials to Faber polynomials by considering the case of circular lunes. As an example, we consider the case where $\alpha = 1/2$. In this case the canonical external conformal map $\Phi$ from the unbounded component of $\C\setminus \sfC_{1/2}$ to the exterior of the closed unit disk has the simple form
\[\Phi(z) = \frac{z^2+1/4}{z}.\]
We therefore find that
\[F_{2n+l}^{\sfC_{1/2}}(z) = \sum_{k=0}^{n}{{2n+l}\choose{k}}4^{-k}z^{2n+l-2k}.\]
For $r\geq1$, we let
\[\sfC_\alpha^r = \{z: |\Phi(z)| = r\} = \left\{\alpha\frac{1+\left(\frac{w-1}{w+1}\right)^\alpha}{1-\left(\frac{w-1}{w+1}\right)^\alpha}: |w| = r\right\}.\]
The computed difference $\|F_{n}^{\sfC_{1/2}}-T_{n}^{\sfC_{1/2}^r}\|_{\infty}$ for $n=11$ is illustrated in Figure \ref{fig-faber}.

\begin{figure}[h!]
\centering
\includegraphics[width=0.8\textwidth]{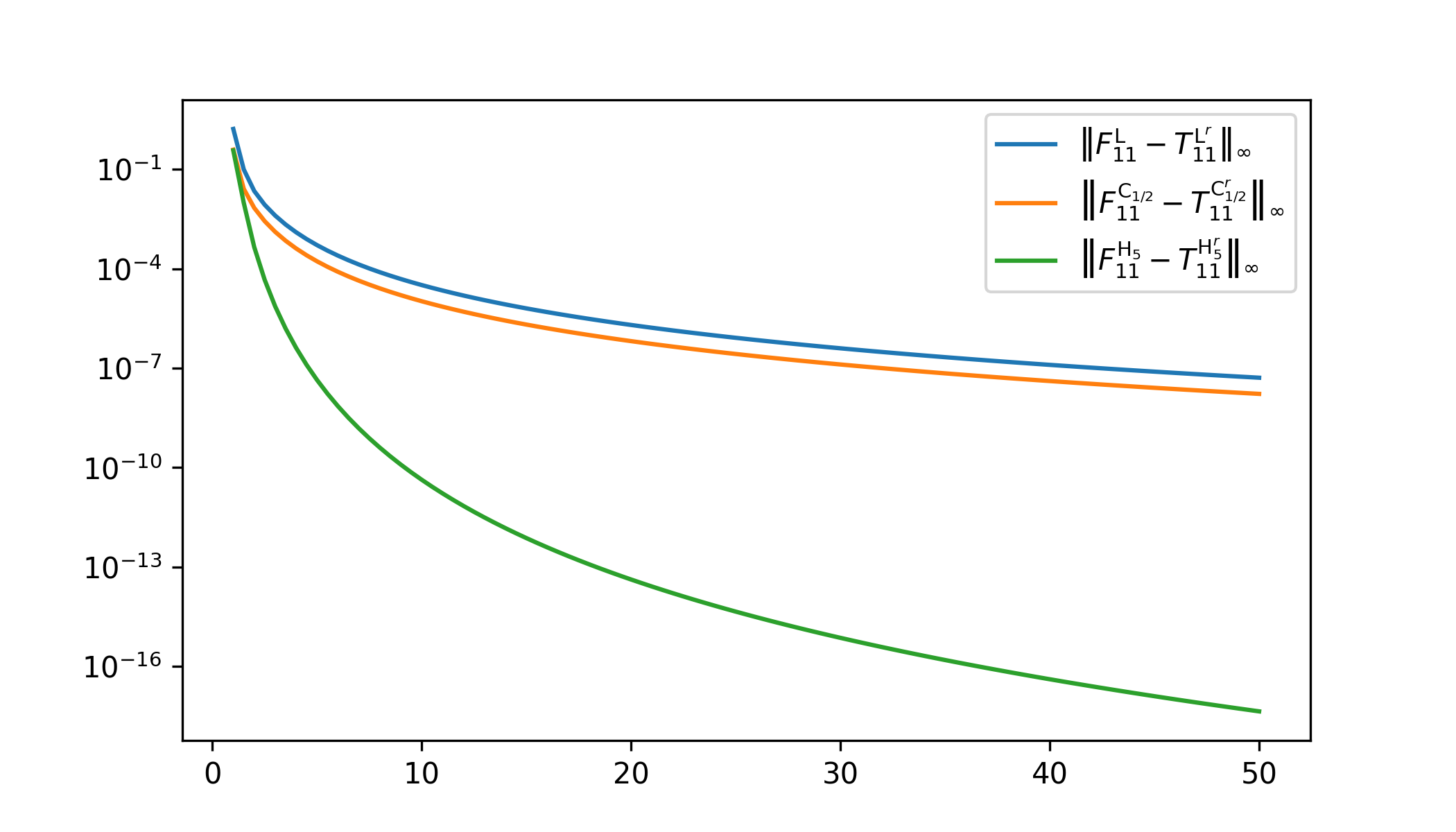}
\caption{\red{The figures illustrate $\|F_{n}^{\sfE}-T_{n}^{\sfE_r}\|_\infty$ as functions of $r>1$ for different $\sfE$ and $n = 11$ using Tang accuracy of $10^{-40}$. For $\sfL^r$ and $\sfC_{1/2}^r$ the estimated convergence rate is around $O(r^{-4})$. For $\sfH_5^{r}$ the estimated convergence rate is $O(r^{-10})$.}}
\label{fig-faber}
\end{figure}

\subsection{Zero distribution}
Our final computations \red{concern} zeros of $T_n^{\sfE}$ for different compact sets $\sfE$. In Figures \ref{fig-zeros-tri}-\ref{fig-zeros-hexagon} the zeros corresponding to $\sfE_\Delta$, $\sfE_\square$, $\sfE_{\pentagon}$ and $\sfE_{\hexagon}$ are \red{displayed}. In Figures \ref{fig-zeros-lune-1-2} and \ref{fig-zeros-lune-3-2} the zeros of $T_n^{\sfC_{\alpha}}$ are illustrated for $\alpha = 1/2$ and $\alpha = 3/2$. In Figures \ref{fig-zeros-del}--\ref{fig-zeros-exo} the zeros of $T_{n}^{\sfH_m}$ are \red{illustrated} for different values of $m$ and $n$. In Figures \ref{fig-zeros-bernoulli} and \ref{fig-zeros-peanut} the zeros corresponding to $T_n^{\sfL}$ and $T_n^{\sfL^r}$ are \red{illustrated}. To complement the plots of the zeros of $T_n^{\sfL^r}$ we also plot the zeros of Chebyshev polynomials corresponding to 
\[\red{\sfA^r\coloneqq\{z: |z^3+z+1| = r\}, \quad \text{and}\quad \sfB^r\coloneqq\{z: |z^4-z^2| = \red{r}\}.}\]
\red{We introduce the constants
\begin{equation}
	r_3 =\sqrt{31/27},\quad \text{ and }\quad r_4 = 1/4
\end{equation}
since $\sfA^{r_3}$ and $\sfB^{r_4}$ correspond to the critical case where the lemniscates have self-intersections.}
The corresponding zero plots are given in Figures \ref{fig-zeros-gen-lem-1} and \ref{fig-zeros-gen-lem-2}.

\begin{figure}[h!]
	\centering
	\begin{subfigure}{.49\textwidth}
		\centering
    \includegraphics[width=\linewidth]{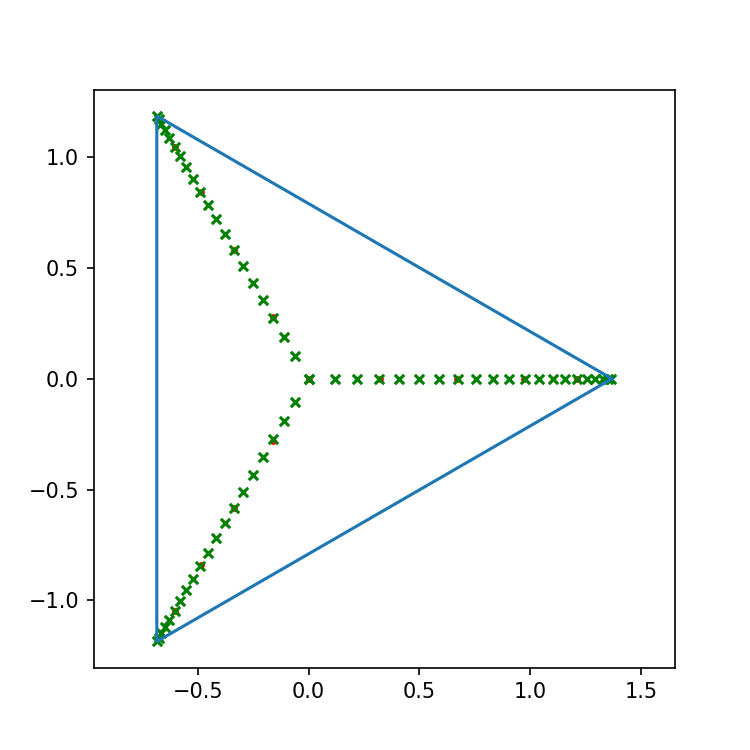}
    \caption{$T_{16}^{\sfE_\Delta}$ and $T_{62}^{\sfE_\Delta}$.}
    \label{fig-zeros-tri}
	\end{subfigure}%
	\begin{subfigure}{.49\textwidth}
		\centering
    \includegraphics[width=\linewidth]{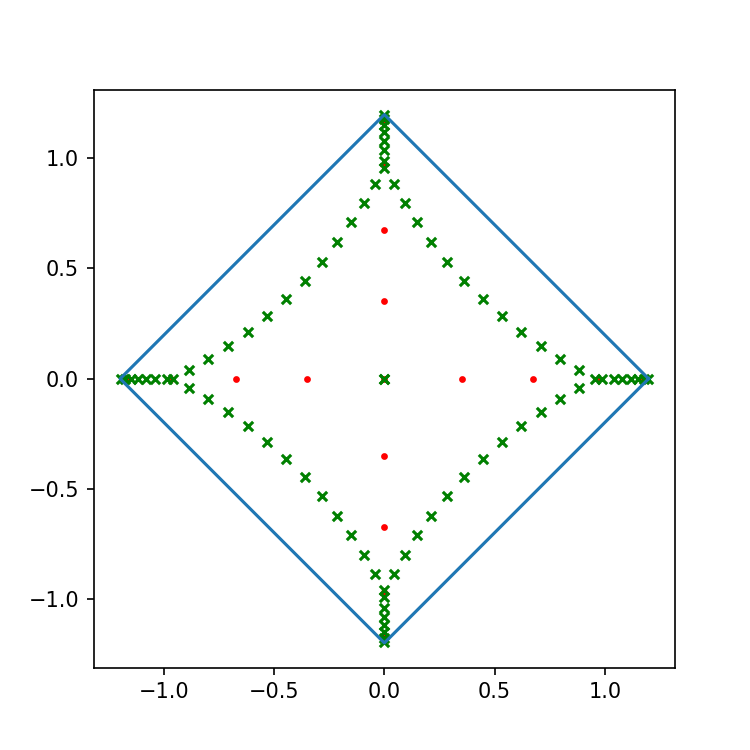}
    \caption{$T_{17}^{\sfE_\square}$ and $T_{82}^{\sfE_\square}$.}
    \label{fig-zeros-square}
	\end{subfigure}
	\begin{subfigure}{.49\textwidth}
		\centering
    \includegraphics[width=\linewidth]{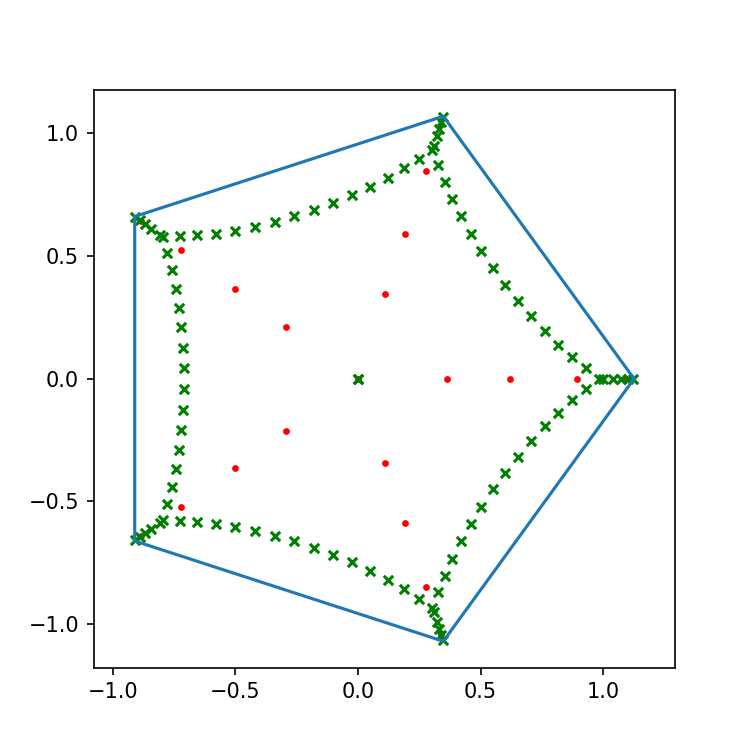}
    \caption{$T_{21}^{\sfE_{\pentagon}}$ and $T_{102}^{\sfE_{\pentagon}}$.}
    \label{fig-zeros-pentagon}
	\end{subfigure}%
	\begin{subfigure}{.49\textwidth}
		\centering
    \includegraphics[width=\linewidth]{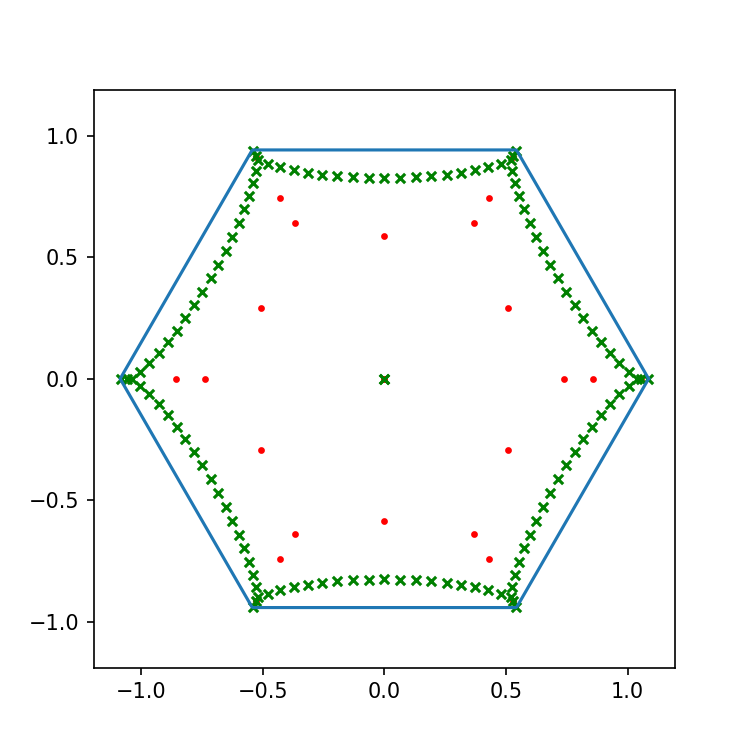}
    \caption{$T_{27}^{\sfE_{\hexagon}}$ and $T_{122}^{\sfE_{\hexagon}}$.}
    \label{fig-zeros-hexagon}
    \end{subfigure}
		\caption{Zeros of Chebyshev polynomials corresponding to $m$-gons. }
\end{figure}
\begin{figure}[h!]
		\begin{subfigure}{.49\textwidth}
		 \centering
    \includegraphics[width=\linewidth]{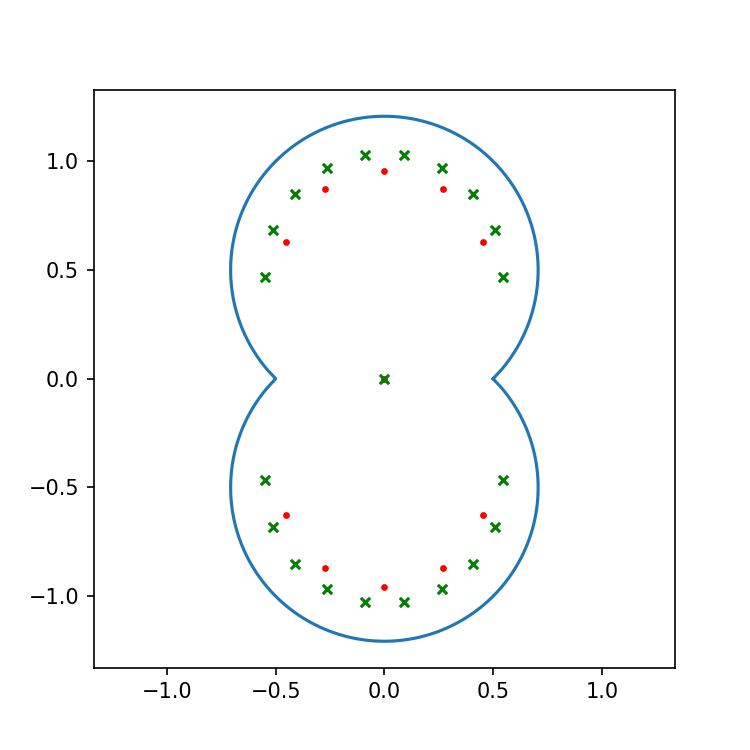}
    \caption{$T_{11}^{\sfC_{1/2}}$ and $T_{21}^{\sfC_{1/2}}$.}
    \label{fig-zeros-lune-1-2}
	\end{subfigure}%
	\begin{subfigure}{.49\textwidth}
		\centering
    \includegraphics[width=\linewidth]{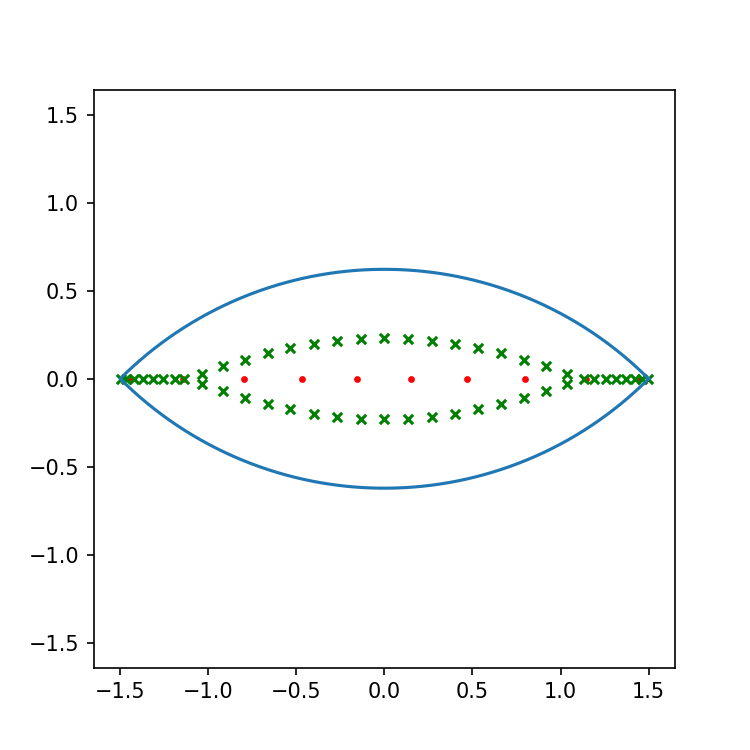}
    \caption{$T_{10}^{\sfC_{3/2}}$ and $T_{50}^{\sfC_{3/2}}$.}
    \label{fig-zeros-lune-3-2}
	\end{subfigure}
			\caption{Zeros of Chebyshev polynomials corresponding to circular lunes. }
\end{figure}

\begin{figure}[h!]
	\begin{subfigure}{.49\textwidth}
		\centering
    \includegraphics[width=\linewidth]{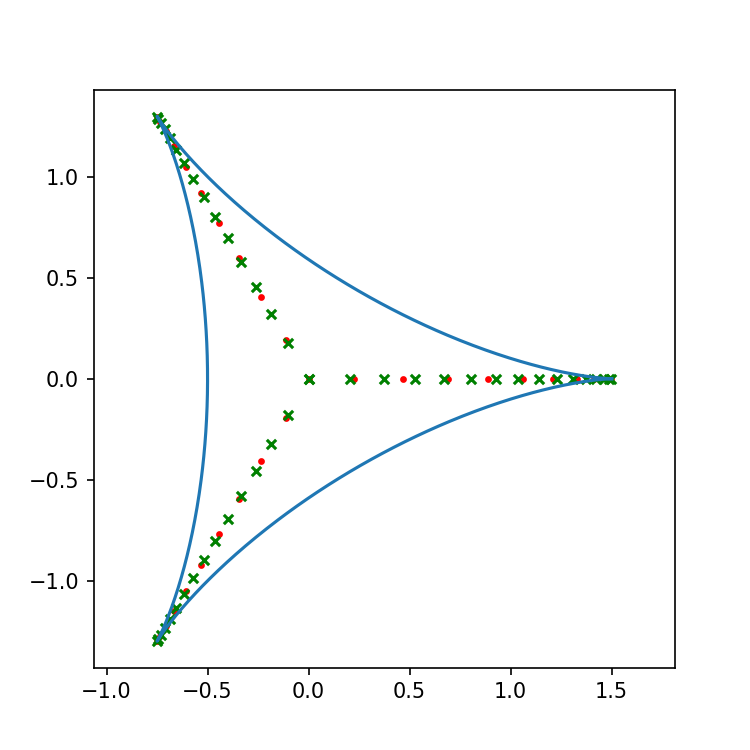}
    \caption{$T_{31}^{\sfH_3}$ and $T_{47}^{\sfH_3}$.}
    \label{fig-zeros-del}
	\end{subfigure}%
	\begin{subfigure}{.49\textwidth}
		 \centering
    \includegraphics[width=\linewidth]{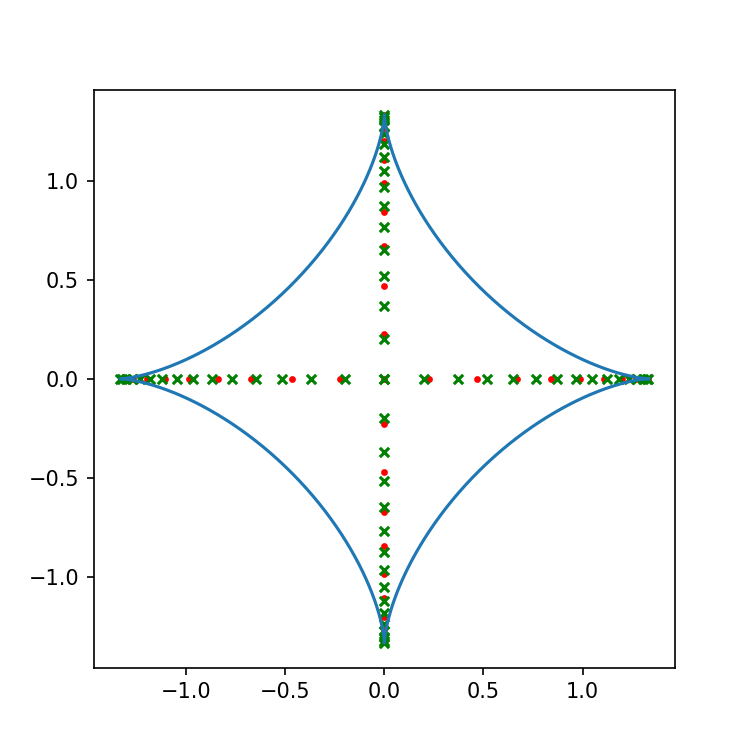}
    \caption{$T_{41}^{\sfH_4}$ and $T_{62}^{\sfH_4}$.}
    \label{fig-zeros-ast}
	\end{subfigure}
	\begin{subfigure}{.49\textwidth}
		\centering
    \includegraphics[width=\linewidth]{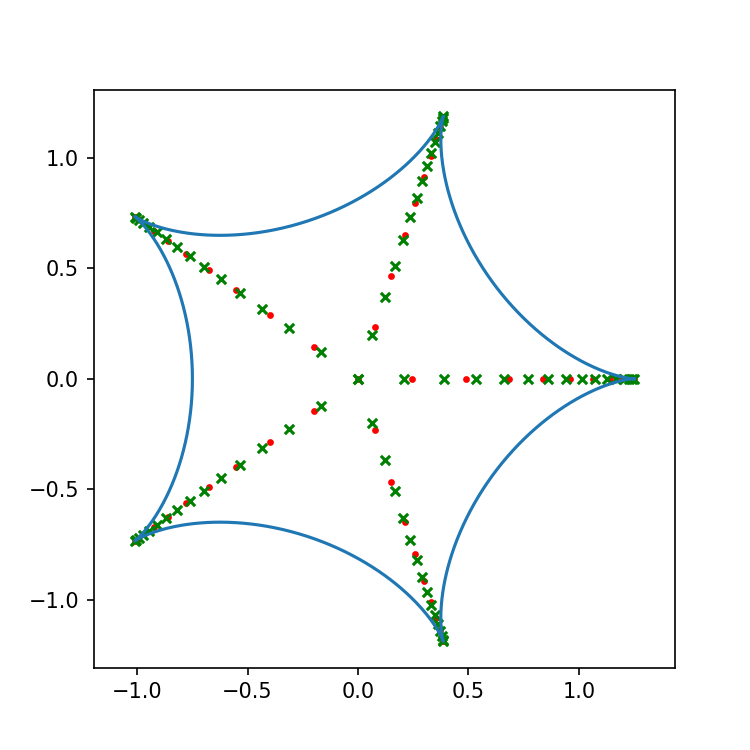}
    \caption{$T_{51}^{\sfH_5}$ and $T_{77}^{\sfH_5}$.}
    \label{fig-zeros-pent}
	\end{subfigure}
	\begin{subfigure}{.49\textwidth}
		\centering
    \includegraphics[width=\linewidth]{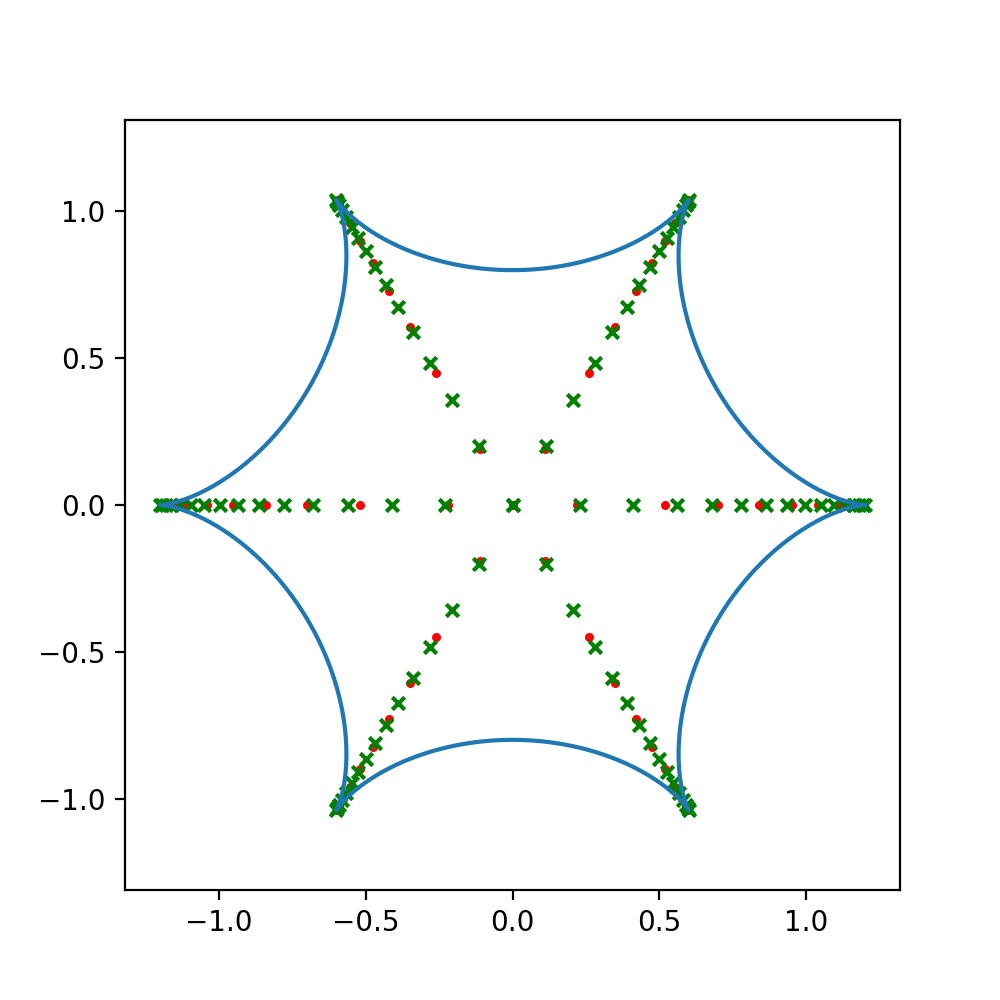}
    \caption{$T_{61}^{\sfH_6}$ and $T_{92}^{\sfH_6}$.}
    \label{fig-zeros-exo}
	\end{subfigure}
	\caption{Zeros of Chebyshev polynomials corresponding to $m$-cusped Hypocycloids. }
\end{figure}
\begin{figure}[h!]
	\begin{subfigure}{.49\textwidth}
		\centering
    \includegraphics[width=\linewidth]{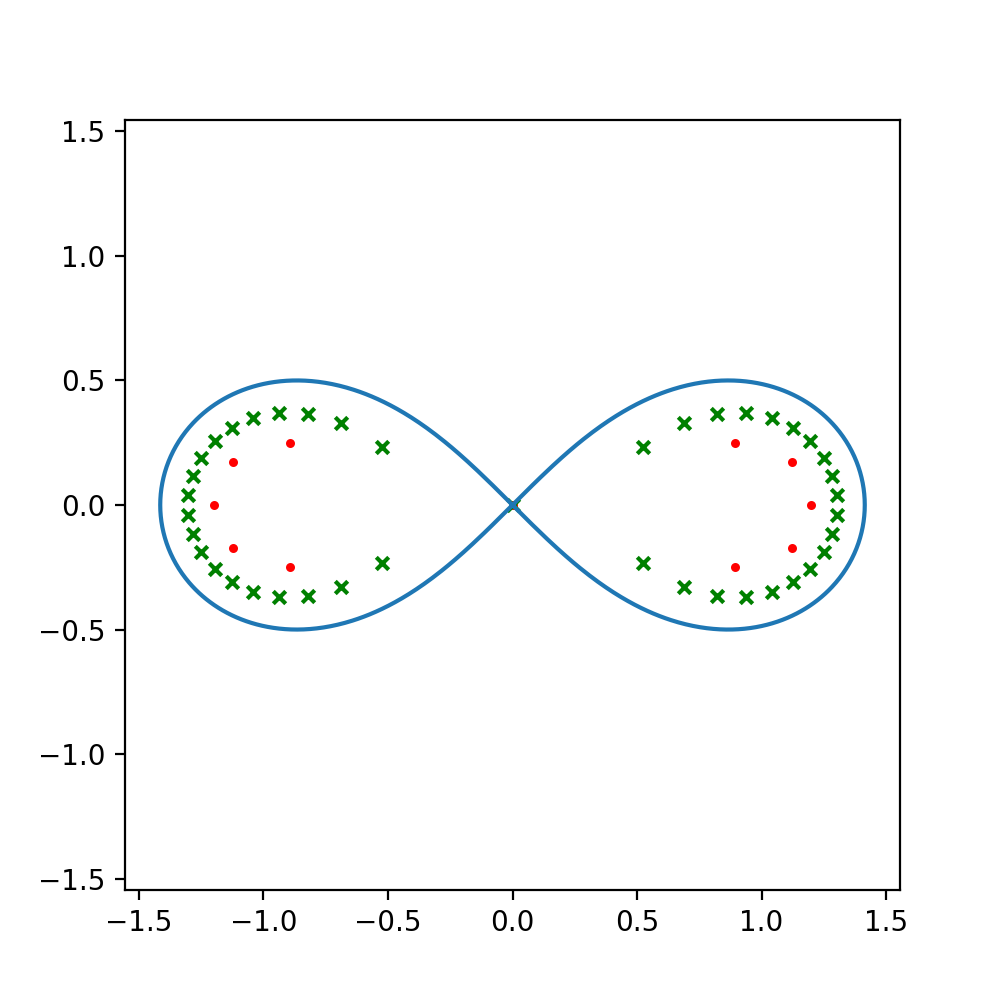}
    \caption{$T_{11}^{\sfL_2}$ and $T_{41}^{\sfL_2}$.}
    \label{fig-zeros-bernoulli}
	\end{subfigure}%
	\begin{subfigure}{.49\textwidth}
		\centering
    \includegraphics[width=\linewidth]{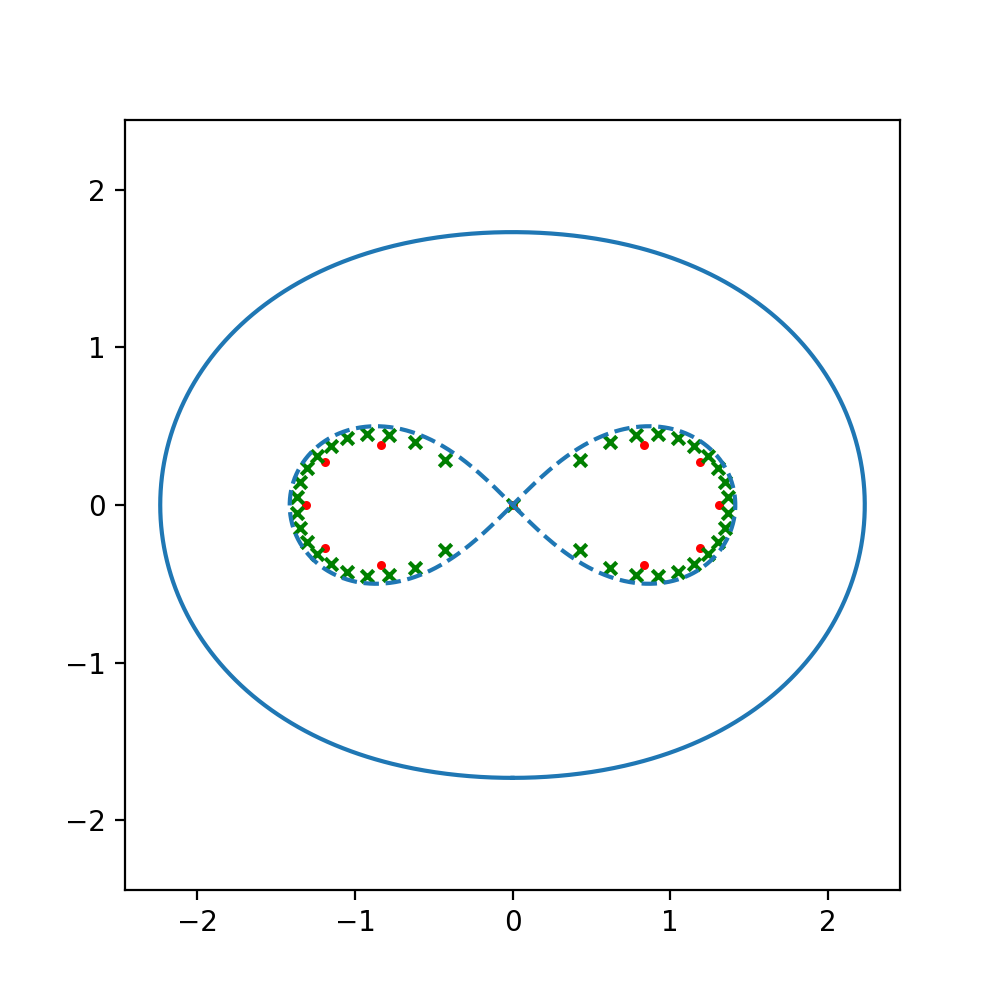}
    \caption{$T_{11}^{\sfL_2^2}$ and $T_{41}^{\sfL_2^2}$. Dotted: $\sfL_2$.}
    \label{fig-zeros-peanut}
	\end{subfigure}
	\begin{subfigure}{.49\textwidth}
		\centering
    \includegraphics[width=\textwidth]{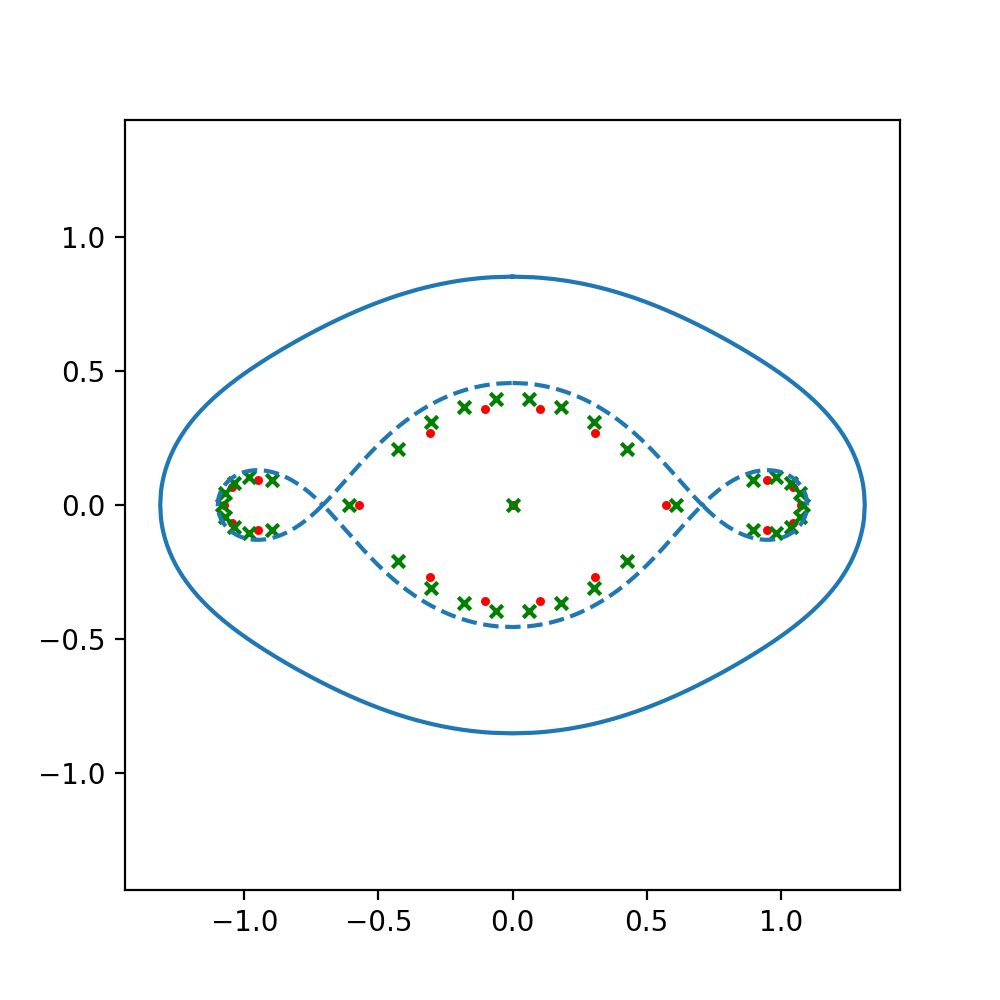}
    \caption{$T_{21}^{\sfB^{5/4}}$ and $T_{37}^{\sfB^{5/4}}$. Dotted: $\sfB^{r_4}$.} 
    \label{fig-zeros-gen-lem-1}
	\end{subfigure}%
	\begin{subfigure}{.49\textwidth}
		\centering
	\includegraphics[width = \textwidth]{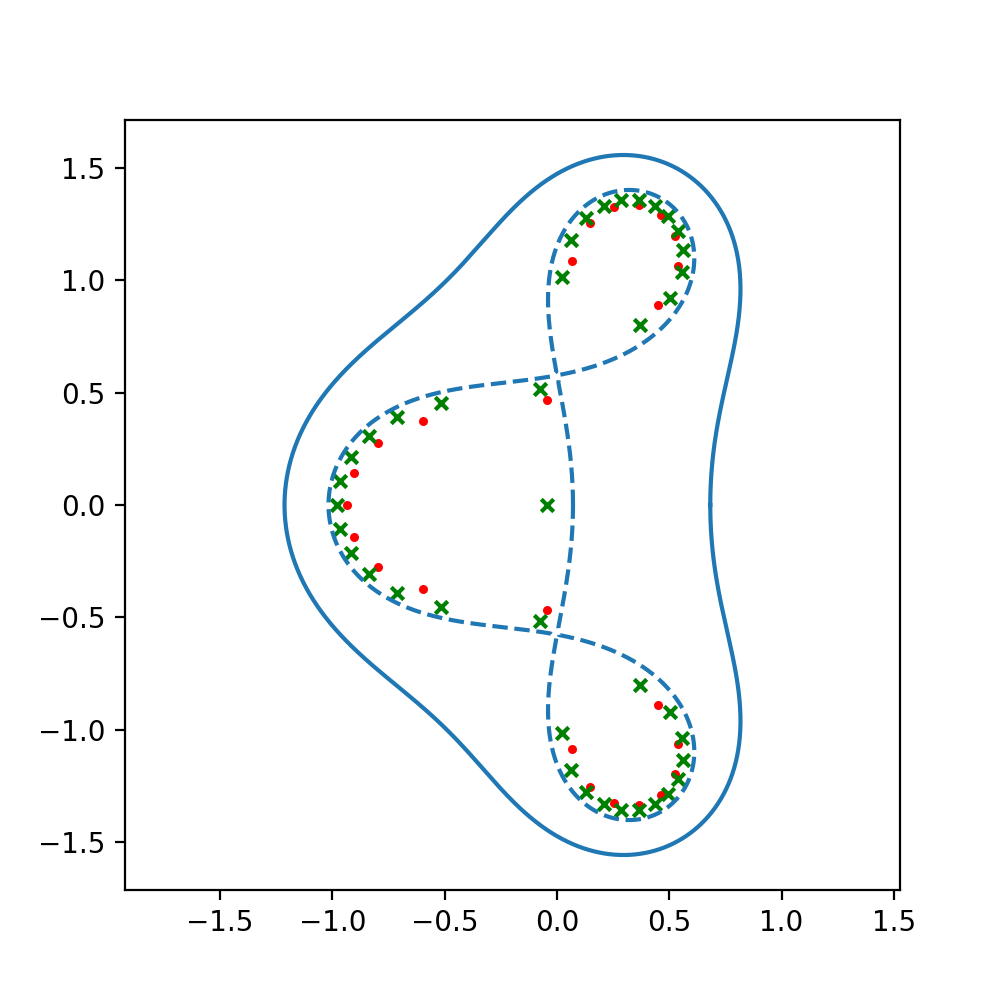}
	\caption{$T_{25}^{\sfA^{2}}$ and $T_{40}^{\sfA^{2}}$. Dotted: $\sfA^{r_3}$.}
	\label{fig-zeros-gen-lem-2}
	\end{subfigure}%
	\caption{Zeros of Chebyshev polynomials corresponding to lemniscates. 
	}
\end{figure}

\section{Discussion}
\label{sec:discussion}
In Section \ref{sec:computations} we saw several examples of computations of Chebyshev polynomials that we here wish to discuss further. 
\subsection{Widom factors}
The Widom factors computed in Section \ref{sec:computations} are computed \red{with an accuracy of $10^{-10}$ within Tang's algorithm}. We believe that Tang's algorithm can be very useful in getting suggesting the behavior of the Widom factors corresponding to a set. This method has previously been applied in \cite{christiansen-eichinger-rubin24} where a result on the limits of Widom factors -- first conjectured using numerical experiments -- was resolved theoretically. \red{We wish to argue for the validity that
	\[\lim_{n\rightarrow \infty}\cW_n(\sfE)=1\]
	holds for closures of Jordan domains with piecewise-analytic boundary.
}

\subsubsection{Regular polygons}
We begin by discussing the Widom factors computed for the regular polygons. \red{As we previously remarked, it is known that the Widom factors are bounded by $2$. Apart from this bound, not much is known.} The plots in Figures \ref{fig-triangle}-\ref{fig-hexagon} clearly suggests that $\{\cW_n(\sfE)\}$ is monotonically decreasing in $n$ if $n>2$ and $\sfE$ is an $m$-gon. \red{The numerics suggests that the Widom factors converge to $1$ at a rate proportional to $1/n$.} \red{This observation points to a difference in the supremum norm between the Chebyshev polynomials and the Faber polynomials associated with the regular $m$-gon.} Indeed, by \cite[Theorem II.2.1]{minadiaz06}, we see that if $\sfE_m$ is an $m$-gon with corners at $\exp(\frac{2\pi ik}{m})$ then
\[|F_n^{\sfE_m}(e^{\frac{2\pi ik}{m}})| =  \left(\frac{2+m}{m}+O(n^{-\frac{2+m}{m}})\right)\Capacity(\sfE_m^n)\]
as $n\rightarrow \infty$ for $k = 0,1,2\dotsc,m-1$. \red{In particular,}
\[\liminf_{n\rightarrow \infty}\frac{\|F_n^{\sfE_m}\|_{\sfE_m}}{\Capacity(\sfE_m)^n}\geq\frac{2+m}{m}.\]
\red{We remark that the dotted lines visible in Figures \ref{fig-triangle}-\ref{fig-hexagon} represent the value $(2+m)/m-1$.} If we choose to believe that \red{the sequence} $\{\cW_n(\sfE_m)\}$ \red{eventually} decreases monotonically for \red{$n$ large enough} then as Figures \ref{fig-triangle}-\ref{fig-hexagon} illustrate, the norms of the Chebyshev polynomials are significantly smaller.

Based on these considerations, the Faber polynomials corresponding to the regular polygons presumably do not provide good enough estimates as trial polynomials to determine the limits of the Widom factors. In short, we believe that the sequence $\{\cW_n(\sfE_m)\}$ \red{eventually} decreases monotonically and that the limit is $1$ as $n\rightarrow \infty$.

One approach in proving that the limit value is $1$ is to analyze some well-suited family of trial polynomials whose normalized norms converge to $1$. How to construct such a family is not immediately clear to us. Under the assumption that $\lim_{n\rightarrow \infty}\cW_n(\sfE_m) = 1$ holds this would not constitute the only example where the Faber polynomials are ill-suited trial polynomials for determining the detailed behavior of $\cW_{n}(\sfE)$. In the extreme case, an example of Clunie \cite{clunie59} further studied by Suetin \cite[p. 179]{suetin84} and Gaier \cite{gaier99} \red{shows} the existence of a quasi-disk $\sfE$ such that the quantity
\[\frac{\|F_n^{\sfE}\|_{\sfE}}{\Capacity(\sfE)^n}\]
is unbounded in $n$ along some sparse subsequence. In comparison, \cite[Theorem 1]{andrievskii-nazarov19} shows that $\cW_{n}(\sfE)$ is still bounded in this case. 

\subsubsection{Hypocycloids}

Recall that $\sfH_m$ denotes the $m$-cusped hypocycloid defined in \eqref{eq:hypocycloid}. Since $\sfH_m$ is piecewise-analytic away from the cusp points which are outward pointing, \cite[Theorem 2.1]{totik-varga15} can be applied to deduce that $\cW_n(\sfH_m)$ is bounded. The Faber polynomials again seem ill-suited in order draw conclusions concerning the precise behavior of the Widom factors in this case\red{,} since it is shown in \cite{he-saff94} that
\[\limsup_{n\rightarrow \infty}\|F_n^{\sfH_m}\|_{\sfH_m}\geq 2\] 
for $m = 2,3,4$. Comparisons with Faber polynomials are therefore inconclusive as to whether
\[\limsup_{n\rightarrow \infty}\cW_{n}(\sfH_m)\leq 2\]
holds or not. The numerical experiments illustrated in Figures \ref{fig-hyp-3}-\ref{fig-hyp-6} paint a richer picture. \red{Again, it seems that the sequence $\cW_{n}(\sfH_m)$ converges to $1$ with a convergence rate proportional to $n^{-1/2}$ and that the sequence $\cW_n(\sfH_m)$ is monotone decreasing if $n$ is large enough. In comparison to the Widom factors of the regular $m$-gons, the decay appears to be slower in the case of hypocycloids. } 

\subsubsection{Circular lunes}
Recall that $\sfC_\alpha$, defined in \eqref{eq:lune_def}, denotes the circular lune with vertices at $\pm \alpha$ and exterior angle $\pi\alpha$. Based on the plots in Figures \ref{fig-lune-1-2} and \ref{fig-lune-3-2} together with the \red{results} in Table \ref{tab:circular-lune}, it seems likely that the Widom factors corresponding to $\sfC_\alpha$ converges to $1$. It is interesting to note that when the set is convex then the whole sequence $\cW_n(\sfC_\alpha)$ appears to be monotonically decreasing, see Figure \ref{fig-lune-3-2}. On the other hand, if $\alpha\in (0,1)$ then two distinct monotonically decreasing subsequences of $\cW_{n}(\sfC_\alpha)$ emerge based on the parity of the degrees.  We believe that the sequence $\{\cW_{2n+l}(\sfC_\alpha)\}_n$ is monotonically decreasing to $1$ for fixed $l\in \{0,1\}$ if $\alpha\in (0,2)$. \red{The rate of convergence appears to be different depending on the convexity properties of $\sfC_{\alpha}$. In particular, if the set is convex then the rate of decrease is proportional to $n^{-1}$ while for $\sfC_{1/2}$ the decay is slower.} If $\alpha = 2$ then $\cW_{n}(\sfC_{2}) = 2$ for any value of $n$. 

\subsection{Motivating the Faber connection}
\label{subsec:faber_connection}
The Chebyshev polynomials and Faber polynomials will both \red{inherit} symmetric structure\red{s from} the underlying set. To see this, one should compare Lemma \ref{lem:symmetry} to \cite[Theorem 2.2]{he94} or \cite[Theorem 2.1]{he-94-2}. This comparison is essentially encapsulated in the following simple lemma.

\red{
\begin{lemma}
    If $\sfE$ is invariant under rotations by $2\pi/m$, $n\in \N$, and $l\in \{0,1,\dotsc,m-1\}$, then both $T_{nm+l}^{\sfE}$ and $F_{nm+l}^{\sfE}$ belong to
    \begin{equation}
    	\{z^lQ(z^m): Q\text{ is a monic polynomial of degree $n$}\}.
    	\label{eq:symmetric_polynomial_set}
    \end{equation}
    In particular,
    \[T_{l}^{\sfE}(z) = F_{l}^{\sfE}(z) = z^l,\quad 0\leq l \leq m-1.\]
    \label{lem:faber-chebyshev-equality}
\end{lemma}
}
\begin{proof}
 \red{From Lemma \ref{lem:symmetry} we conclude the validity of every claim made on $T_{nm+l}^{\sfE}$. The conformal map $\Phi$ from \eqref{eq:exterior-conformal} necessarily satisfies}
    \[e^{-2\pi i/m}\Phi(e^{2\pi i/m}z) = \Phi(z)\]
    from which we gather that
    \[F_{nm+l}^{\sfE}(e^{2\pi i/m}z) = e^{2\pi il/m}F_{nm+l}^{\sfE}(z) \]
    if $l\in \{0,1,\dotsc,m-1\}$. \red{In a completely analogous fashion to how Lemma \ref{lem:symmetry} was shown, we gather that $F_{nm+l}^{\sfE}$ belongs to the set in \eqref{eq:symmetric_polynomial_set}.} In the special case where $n = 0$ we see that $F_l^{\sfE}(z) = z^l\red{ = T_{l}^{\sfE}}(z)$. 
\end{proof}
Of course Lemma \ref{lem:faber-chebyshev-equality} has more to do with the rotational symmetry of a set than any other property. It does, however, give several easy examples where the two families of polynomials overlap. 

If $\sfE$ is a rectifiable Jordan curve and $\Phi$ is the conformal map from the exterior of $\sfE$ to $\{z:|z|>1\}$ of the form
\[\Phi(z) = \Capacity(\sfE)^{-1}z+a_0+a_{-1}z^{-1}+\cdots\]
then it can be shown that
\[F_n^{\sfE}(z) = [\Capacity(\sfE)\Phi(z)]^n\left(1+O\left(\frac{1}{r^n}\right)\right)\]
for $z\in \sfE^r \coloneqq \{\zeta: |\Phi(\zeta)| = r\}$. If $[\Capacity(\sfE)\Phi(z)]^n$ was a polynomial of degree $n$ then it would follow from \eqref{eq:szego_inequality} that it would coincide with the corresponding Chebyshev polynomial. We have already seen examples of this when studying lemniscates. Although this is rarely the case, we observe that $F_n^{\sfE}$ will be an increasingly good candidate for obtaining relatively small maximum values on $\sfE^r$ as $r\rightarrow \infty$. For a fixed degree $n$, $F_n^{\sfE}$ will be asymptotically minimal on $\sfE^r$ in the sense that
\[\lim_{r\rightarrow \infty}\frac{\|F_n^{\sfE}\|_{\sfE^r}}{\Capacity(\sfE^r)^n} =1.\]
We believe that this serves as motivation for why one could expect
\begin{equation}
	T_n^{\sfE^r}\rightarrow F_n^{\sfE}
	\label{eq:chebyshev_faber_limit}
\end{equation}
as $r\rightarrow \infty$ to hold in general. Based on the numerical data illustrated in Figure \ref{fig-faber} this is clearly hinted upon for these specific domains. 

\red{We} remark that similar patterns have materialized for any other combination of degrees and sets that we have considered. In the general case\red{,} it is clear that $\sfE^r$ will be an analytic curve for $r>1$ and hence the regularity of the boundary of $\sfE$ is perhaps of less importance since the Faber polynomials corresponding to $\sfE$ are the same as the ones corresponding to $\sfE^r$. We stress the fact that the algorithm outputs polynomials $P_n$ such that $\|P_n\|_{\sfE}-\|T_n^{\sfE}\|_{\sfE}$ is small. This is not exactly the same as saying that $\|P_n-T_n^{\sfE}\|_{\infty}$ is small with $\|\cdot\|_\infty$ defined in \eqref{eq:polynomial_norm}. What is true, is that for a fixed $n$, $\|P_n\|_{\sfE}\rightarrow \|T_n^{\sfE}\|_{\sfE}$ implies that $\|P_n-T_n^{\sfE}\|_{\infty}\rightarrow 0$. The computations remain consistent throughout. No matter how close we approximate the minimal norm, the behavior as suggested in Figure \ref{fig-faber} remains.

\subsection{Zero distribution}
We recall that if $P$ is a polynomial then $\nu$ is the probability measure defined in Section \ref{sec:intro} via the formula
\[\nu(P) = \frac{1}{\deg(P)}\sum_{j=1}^{\deg(P)}\delta_{z_j}\red{,}\]
where $\{z_j\}$ are the zeros of $P$ counting multiplicity. Also, given a compact set $\sfE$ we use $\mu_\sfE$ to denote the equilibrium measure on $\sfE$.

It is shown in \cite{saff-totik90} that the zeros \red{of Chebyshev polynomials} corresponding to the closure of a Jordan domain stay away from the boundary precisely when the bounding curve is analytic. As such we see that in all of our examples, except for the cases of lemniscates $\{z:|P(z)| = r\}$ with analytic boundary, the zeros should approach some part of the boundary. From \cite[Theorem 1.1]{christiansen-simon-zinchenko-IV} we gather that every ``corner point'' on the respective sets $\sfC_\alpha$, $\sfE_m$ and $\sfH_m$ attract zeros. This also appears to be the case, albeit, slowly for $\sfC_{1/2}$.

Predicting the behavior of zeros of extremal polynomials based on plots has proven hazardous in the past. In particular, we refer to the reader to \cite{saff-stylianopoulos08} where five conjectures concerning limiting zero distributions are made very plausible using numerical plots only to be proven to be wrong using theoretical results. 

\subsubsection{Regular polygons}
We adopt the notation $\sfE_m$ to denote the regular polygon with $m$ sides. As is suggested by Figures \ref{fig-zeros-tri}-\ref{fig-zeros-hexagon}, the zeros of $T_n^{\sfE_m}$ for low degrees appear to lie on the diagonal lines between the vertices and the origin. However, by increasing the degree it seems clear that the zeros approach the boundary in all cases but \red{the equilateral triangle as seen in \ref{fig-zeros-tri}}. In \cite{saff-stylianopoulos08}\red{,} the case of Faber polynomials on $\sfE_3$ are discussed. Here the authors specify that for small degrees the zeros of $F_n^{\sfE_3}$ appear to distribute along the diagonals\red{,} however\red{,} they also note that as a consequence of \cite[Theorem 1.5]{kuijlaars-saff95} at least a subsequence of $\nu(F_n^{\sfE_3})$ converges in the weak-star sense to $\mu_{\sfE_3}$ which is supported on the boundary. \red{We find it intriguing to determine whether the zeros of the Chebyshev polynomials on $\sfE_3$ remain on the diagonal for all values of $n$ as this contrasts the behavior of Faber polynomials in this particular case.}

The zeros of $F_n^{\sfE_m}$ for general $m$-gons are illustrated in \cite{he-94-2} and appear to behave very similar to the ones for $T_n^{\sfE_m}$ \red{when $m\geq 4$}. We therefore believe that the zeros should approach the boundary \red{when $m\geq 4$} in the sense that \eqref{eq:zeros_inside} should hold for every compact set in the interior. This would of course also imply that
\[\nu(T_n^{\sfE_m})\xrightarrow{\ast}\mu_{\sfE_m}\]
as $n\rightarrow \infty$.

\subsubsection{Circular lunes}
Recall the definition of $\sfC_{\alpha}$ from \eqref{eq:lune_def}. Based on the plot in Figure \ref{fig-zeros-lune-1-2} it appears \red{that} most of the zeros approach the boundary in the case when $\alpha = 1/2$ and that \eqref{eq:zeros_inside} should hold for any compact set contained in the interior. This is in fact a known result and follows from \cite[Theorem 2.1]{saff-stylianopoulos15}. Indeed, from there we gather that $\nu(T_n^{\sfC_{\alpha}})\xrightarrow{\ast}\mu_{\sfC_\alpha}$ as $n\rightarrow \infty$ for any $\alpha\in (0,1)$. In this sense\red{,} the computed polynomials serve to confirm the predicted \red{behavior} from theoretical results. For any value of $\alpha\in (0,1)\cup(1,2)$ it follows from \cite[Theorem 1.5]{kuijlaars-saff95} that $\nu(F_n^{\sfC_{\alpha}})\xrightarrow{\ast}\mu_{\sfC_\alpha}$ along some subsequence. Motivated by the resemblance between the plots of zeros for Faber polynomials in \cite{he-94-2} with the corresponding zeros of $T_n^{\sfC_\alpha}$ computed here, we suspect that \[\nu(T_n^{\sfC_\alpha})\xrightarrow{\ast}\mu_{\sfC_\alpha},\quad n\rightarrow \infty\] for any value of $\alpha\in (0,1)\cup(1,2)$. Note that $\sfC_{1} = \T$ and hence $T_n^{\sfC_{1}}(z) =z^n$ has all its zeros at the origin.

\subsubsection{Hypocycloids}
\red{If the bounding curve has an outward cusp\red{,} then the example of Hypocycloids illustrate that the zeros of Faber polynomials need not approach all of the boundary, as shown in \cite{he-saff94}.} We believe that the \red{zeros of the corresponding Chebyshev polynomials behave analogously}. It is clearly suggested by Figures \ref{fig-zeros-del}-\ref{fig-zeros-exo} that the support of $\nu(T_n^{\sfH_m})$ is confined to the diagonals between the cusps and the origin. This is in accordance with the behavior exhibited by $\nu(F_n^{\sfH_m})$ and we believe that an analogous result as \cite[Theorem 3.1]{he-saff94} is true in this case. 

\red{
That is to say, we believe that the zeros of $T_n^{\sfH_m^r}$ for $r\geq 1$ are confined to the set
\begin{equation}
	\left\{te^{2\pi i k/m}:0\leq t \leq \frac{m}{m-1},\, k = 0,1,\dotsc,m-1\right\}.
\label{eq:hyp-lines}
    \end{equation}}
If we choose to believe the connection \red{between Chebyshev polynomials and Faber polynomials on equipotential lines from \eqref{eq:chebyshev_faber_limit}} then this would imply that the zeros of $T_n^{\sfH_m^r}$ would move along the \red{diagonals \eqref{eq:hyp-lines}} as $r$ increases and approach the corresponding zeros of the Faber polynomials, this is something we find reasonable to believe.

On the other hand, we note that numerical simulations indicate that the zeros of the corresponding Bergman polynomials corresponding to $\sfH_m$ and its interior, all lie on the straight lines in \eqref{eq:hyp-lines} for small degrees. However\red{,} it follows from \cite[Theorem 2.1]{saff-stylianopoulos08} that at least a subsequence of the Bergman polynomials have zero counting measures converging weak-star to $\mu_{\sfH_m}$.

\subsubsection{Lemniscates}
Recall that $\sfL^r_m = \{z:|z^m-1| = r^m\}$, $\sfL^r_2 = \sfL^r$ and that $\sfL = \sfL^1$. Based on Figure \ref{fig-zeros-bernoulli} it seems reasonable to assume that 
\begin{equation}
    \lim_{n\rightarrow \infty}\nu(T_{2n+1}^{\sfL})(M) = 0
    \label{eq:bernoulli-compact-set-zeros}
\end{equation}
for any compact set $M$ contained in $\{z:|z^2-1|<1\}$. It actually appears to be the case that all the zeros approach the boundary. The main theorem in \cite{saff-totik90}, which states that zeros of Chebyshev polynomials corresponding to an analytic Jordan curve stay away from the boundary is not applicable in this case because $\sfL$ does not have a connected interior. If \eqref{eq:bernoulli-compact-set-zeros} could be established, a consequence of this would be that $\nu(T_{2n+1}^{\sfL})$ converges in the weak-star sense to the equilibrium measure on $\sfL$. It should be noted in this regard that by changing the variable to $\zeta = z^2-1$ it follows that
\[T_{2n+1}^{\sfL}(z) = (\zeta+1)^{1/2}T_n^{1/2}(\zeta)\]
where $T_n^{1/2}$ is the monic minimizer of the expression
\[\max_{\zeta\in \T}\left|(\zeta+1)^{1/2}\left(\zeta^n+\sum_{k=0}^{n-1}a_k\zeta^k\right)\right|.\]
Corresponding to each weight of the form $|\zeta+1|^{s}$ for $s\geq 0$ there is a minimizing weighted Chebyshev polynomial which we denote with $T_n^{s}$, see \cite{bergman-rubin24}. In the particular case where $s = 1$ it is shown in \cite[Theorem 3]{bergman-rubin24} that $\nu(T_n^{1})$ converges \red{in the } weak-star \red{sense} to \red{the} equilibrium measure on $\T$. This implies that an analogous result as \eqref{eq:bernoulli-compact-set-zeros} is valid for compact subsets of $\D$. There is no reason to believe that such a result should exclusively hold for the parameter value of $s=1$ and we therefore suspect that
    \[\nu(T_{2n+1}^{\sfL})\xrightarrow{\ast}\mu_{\sfL}, \quad n\rightarrow \infty.\]

Note that $\nu(T_{2n}^{\sfL}) = \frac{1}{2}(\delta_{-1}+\delta_1)$ for any $n$ and hence very different zero behavior would be exhibited for the different subsequences if the conjecture is true. This is however the case for the Faber polynomials. From a result in \cite{ullman60}, it follows that
\[\nu(F_{2n+1}^{\sfL})\xrightarrow{\ast}\mu_{\sfL}.\]
Furthermore, it is shown there that all the zeros of $F_{2n+1}^{\sfL}$ lie on or inside $\sfL$.

We turn our attention to the outer lemniscates $\sfL^r$ with $r>1$. Surprisingly, based on Figure \ref{fig-zeros-peanut} it seems like the zeros of $T_{2n+1}^{\sfL^r}$ all lie strictly inside $\sfL$ except for the single zero at $0$. Although the main Theorem in \cite{saff-totik90} implies that the zeros asymptotically stay away from $\sfL^r$\red{,} there is no results hinting toward the fact the zeros seem to cluster on $\sfL$. If one believes in \eqref{eq:chebyshev_faber_limit}, \red{as argued for in Section \ref{subsec:faber_connection},} \red{we find it} reasonable to \red{believe} that the zeros of $T_{2n+1}^{\sfL^r}$ lie on or inside $\sfL$ for all values of $n$ and $r$ since the zeros of $F_{2n+1}^{\sfL}$ have this very behavior.
    
Analogous results seem to hold true with $\sfL^r$ replaced by $\sfL_m^r$ for any value of $m$ as the corresponding numerical simulations indicate the same pattern. Generalizations of Ullmans result concerning the asymptotic zero distribution of the Faber polynomials on $\sfL_m$ can be found in \cite{he94}. 

We further believe that a \red{similar picture} holds for any connected lemniscate. To \red{make precise our belief,} we introduce the notion of a critical value of a polynomial. This is a number $P(z)$ where $z$ is such that $P'(z) = 0$. The polynomial $z^2-1$ has one critical value, namely $-1$ which is attained at the origin. This implies that the curve $\sfL = \{z: |z^2-1| = 1\}$ will contain a critical point of $z^2-1$ resulting in the fact that the curve forms a crossing with itself at the origin. In general, if $c$ is a critical value of a polynomial $P$ then $\{z:|P(z)| =|c|\}$ will contain a crossing point.

If we consider the polynomial $Q(z) = z^4-z^2$ then $Q$ has two critical values, namely $1/4$ and $0$. Upon inspection of Figure \ref{fig-zeros-gen-lem-1} it becomes apparent that the zeros of the Chebyshev polynomials on the curve $\{z: |z^4-z^2| = 5/4\}$ seem to approach the critical curve $\{z: |z^4-z^2| = 1/4\}$ which correspond to the lemniscate where the largest critical value is attained (in modulus). Equivalently, this curve is characterized by being the curve $\{z:|z^4-z^2| = r\}$ with smallest value of $r>0$ which is connected. 

A similar pattern emerges for the lemniscates of the form $\{z:|z^3+z+1| = r\}$ with $r\geq \sqrt{31/27}$, see Figure \ref{fig-zeros-gen-lem-2}. For the polynomial $P(z) = z^3+z+1$ the critical point is $\pm i/\sqrt{3}$ and the corresponding critical value is $1\pm i2/3\sqrt{3}$. Since $|1+i2/3\sqrt{3}| = \sqrt{31/27}$ we see that the critical lemniscate corresponds to $r = \sqrt{31/27}$. Again, this critical lemniscate seems to attract the zeros of the Chebyshev polynomials corresponding to larger values of $r$. 

\red{We question if this holds in greater generality. Let $P$ be a polynomial of degree $m$ with largest critical value in terms of absolute value given by $c$. For any $r\geq|c|$ let
	\[\sfE^r = \{z: |P(z)| = r\}\]
	then for a fixed $l\in \{1,\dotsc,m-1\}$ is it true that
	\[\nu(T_{nm+l}^{\sfE^r})\xrightarrow{\ast}\mu_{\sfE^{|c|}}\]
	as $n\rightarrow \infty$?}

Based on Figures \ref{fig-zeros-gen-lem-1} and \ref{fig-zeros-gen-lem-2} this seems to be the case. Observe that \eqref{eq:cheb_nm_sequence} implies that $T_{nm}^{\sfE^r} = a^{-n}P(z)^n$ where $a$ is the leading coefficient of $P$ in which case the zero counting measure is constant.

It could be further speculated what happens in the general case for level curves of conformal maps. Assume that $\sfE$ is a connected compact set with simply connected complement and $\Phi:\C\setminus\sfE\rightarrow\{z:|z|>1\}$ is the conformal map of the form $\Phi(z)=\Capacity(\sfE)^{-1}z+O(1)$ as $z\rightarrow\infty$. Again, introducing the set $\sfE^r=\{z:|\Phi(z)|=r\}$ then the bounding curve of $\sfE^r$ is analytic for $r>1$. From \cite{saff-totik90} we know that the zeros of $T_n^{\sfE^r}$ asymptotically stay away from the boundary, in the sense that there exists a neighborhood of the boundary where $T_n^{\sfE^r}$ is zero free for large $n$. The question is if something similar as in the case of lemniscates happens in this situation. Do the zeros asymptotically approach $\sfE$? This is true for the corresponding Faber polynomials \red{which} could hint at this being true for the corresponding Chebyshev polynomials.

\subsection{Concluding remarks}
With this article, we hope to exemplify the usefulness of Tang's generalization of the Remez algorithm to the study of Chebyshev polynomials. Our research into the matters commenced by considering the zeros of the Chebyshev polynomials corresponding to the Bernoulli lemniscate
\[\sfL = \{z:|z^2-1| = 1\}.\]
Based on the fact that $T_{2n}^{\sfL}(z) = (z^2-1)^n$ it was suggested in \cite{christiansen-simon-zinchenko-III} that the odd Chebyshev polynomials $T_{2n+1}^{\sfL}$\red{,} which apart from having a zero at the origin\red{,} should behave similarly. Explicitly it is written \red{in} \cite[p. 215]{christiansen-simon-zinchenko-III} that 
\red{\begin{quote}
``...we suspect (but cannot prove) that for $j$ large all the other zeros of $T_{2j+1}$ lie in small neighborhoods of $\pm 1$ and that the above $d\mu_\infty$ is also the limit through odd $n$'s.'' 	
\end{quote}}
Here $d\mu_\infty = \frac{1}{2}(\delta_{-1}+\delta_1)$. We initially set out to show this. Since we did not progress in this regard we started considering numerical methods to compute the Chebyshev polynomials with the intent to better understand how the zeros approached $\pm1$. Using Tang's algorithm we could compute the Chebyshev polynomials corresponding to $\sfL$ and the result surprised us. The zeros seemed to behave opposite to \red{what was conjectured. They approached} the bounding curve rather than the two points $\pm 1$. The use of Tang's algorithm therefore showed us that the hypothesis we initially had believed was probably incorrect and that \red{the} conjecture should be modified. \red{Partial progress to understanding the zeros of Chebyshev polynomials relative to lemniscates was done in} \cite{bergman-rubin24} by studying a related problem. However, we are still lacking a complete proof of this.

With Tang's algorithm at hand\red{,} we set out to study Chebyshev polynomials corresponding to a wide variety of sets whose asymptotic behavior remain unknown. We believe that making use of Tang's algorithm \red{(or other related algorithms such as the Lawson algorithm)} is a good way of getting predictions on the behavior of Chebyshev polynomials. The results in \cite{bergman-rubin24} and \cite{christiansen-eichinger-rubin24} are based on hypothesis formulated using initial numerical experiments. Some rather surprising results have also been suggested to us by numerical experiments along the way. In particular, the relation between Faber polynomials and Chebyshev polynomials does not seem to have been given any attention in the literature\red{, although it is known that they coincide for certain sets}.

In short, we believe that \red{the} use of Tang's algorithm in the study of Chebyshev polynomials may prove useful in the future when formulating \red{well-posed hypothesis} on their asymptotic behavior.

\appendix
\section{Tang's algorithm}
\label{sec:appendix}
We recall that Tang's algorithm seeks a linear functional
\begin{equation}
	L_{\boldr,\boldalpha,\boldz}(g)=\sum_{j=1}^{n+1}r_j\mathrm{Re}(e^{-i\alpha_j}g(z_j))
\end{equation}
conditioned to satisfy $r_j\in [0,1]$, $\alpha_j\in [0,2\pi)$, $z_j\in \sfE$, $\sum r_j = 1$ and $L_{\boldr,\boldalpha,\boldz}(\varphi_k) = 0$ for every $k = 1,\dotsc,n$. The goal with applying the algorithm is to obtain coefficients $\lambda_{1},\dotsc,\lambda_n$ such that
\[\|f-\sum_{k=1}^{n}\lambda_k\varphi_k\|_{\sfE}\]
is minimal.

The linear nature of the maximizing linear functional suggests that it is beneficial to change the perspective to linear algebra. We use the notation from \cite{tang87,tang88, komodromos-russell-tang95,fischer-modersitzki-93} and define the matrix
\begin{equation}
    A(\boldz,\boldalpha) \red{\coloneqq} \begin{pmatrix}
        1 & 1 & \cdots & 1 \\
        \re(e^{-i\alpha_1}\varphi_1(z_1))  & \re(e^{-i\alpha_2}\varphi_1(z_2)) & \cdots & \re(e^{-i\alpha_{n+1}}\varphi_1(z_{n+1})) \\
        \vdots & \vdots & \ddots & \vdots \\
        \re(e^{-i\alpha_1}\varphi_n(z_1))  & \re(e^{-i\alpha_2}\varphi_n(z_2)) & \cdots & \re(e^{-i\alpha_{n+1}}\varphi_n(z_{n+1})) 
    \end{pmatrix}
\end{equation}
together with the vector
\begin{equation}
    c_f(\boldz, \boldalpha) \red{\coloneqq} \begin{pmatrix}
        \re(e^{-i\alpha_1}f(z_1)) \\ 
        \vdots \\
        \re(e^{-i\alpha_{n+1}}f(z_{n+1}))
    \end{pmatrix}.
\end{equation}
It then follows from \eqref{eq:maximizing_linear_functional_form} that
\begin{equation}
    L_{\boldr,\boldalpha,\boldz}(f) = c_f(\boldz, \boldalpha)^T \boldr
    \label{eq:linear_functional_vector_representation}
\end{equation}
and the constraints \eqref{eq:condition_convex} and \eqref{eq:condition_annihilate} become embedded in the equation
\begin{equation}
    A(\boldz,\boldalpha)\boldr = \begin{pmatrix}
    1\\0\\\vdots\\0
\end{pmatrix}.
\label{eq:constraint_matrix}
\end{equation}
Parameters $\boldr,\,\boldalpha,\,\boldz$ satisfying \eqref{eq:constraint_matrix} are called \emph{admisible} if additionally $A(\boldz,\boldalpha)$ is invertible. If $\varphi^\ast = \sum_{k=1}^{n}\lambda_k^\ast \varphi_k$, $\lambda_k\red{^\ast}\in \R$, is a best approximation and $\boldr^\ast,$ $\boldalpha^\ast$, $\boldz^\ast$ are corresponding admissible parameters such that $L_{\boldr^\ast,\boldalpha^\ast,\boldz^\ast}(f) = \|f-\varphi^\ast\|_{\sfE}$ then
\begin{equation}
    A(\boldz^\ast, \boldalpha^\ast )^T\begin{pmatrix}
    \|f-\varphi^\ast\|_{\sfE}\\
    \lambda_1^\ast \\ \vdots \\ \lambda_n^\ast
    \end{pmatrix} = c_f(\boldz^\ast,\boldalpha^\ast)
    \label{eq:get_coefficients}
\end{equation}
and therefore if $A(\boldz^\ast,\boldalpha^\ast)$ is invertible we can recover the extremal coefficients $\lambda_1^\ast,\dotsc,\lambda_n^\ast$ from $A(\boldz^\ast, \boldalpha^\ast )$ and $c_f(\boldz^\ast,\boldalpha^\ast)$. We assume, as in \cite{tang87,tang88,fischer-modersitzki-93} that $\sfE = I\eqqcolon[0,1]$ (which is no restriction) since we can always parametrize $\sfE$ using $[0,1]$. To emphasize that we are working on $[0,1]$ we let $\boldz = \boldt = \{t_j\}$.

\red{An implementation of the algorithm in Python can be found at \url{https://github.com/olguvirubin/complex_chebyshev}}

\begin{algorithm}[h!]
\DontPrintSemicolon
  \KwInput{Basis functions $\{\varphi_k\}_{k=1}^{n}$, $f\in \cC_\R(I)$, $\{t_j^{(1)}\}_{j=1}^{n+1}$, $\{\alpha_j^{(1)}\}_{j=1}^{n+1}, \mathrm{threshold}$}
  \tcp*{The parameters may be taken at random but should satisfy that $A(\boldt^{(1)},\boldalpha^{(1)})$ is invertible.}
  \KwFind{$\boldr$ satisfying $A(\boldt^{(1)},\boldalpha^{(1)})\boldr = \begin{bmatrix}
      1\\
      \boldsymbol{0}
  \end{bmatrix}$.}
  \For{$j\in \{1,\dotsc,n+1\}$}
  {
  \If{$r_j<0$}{
  \KwLet{$\alpha_j^{(1)} \coloneqq \alpha_j^{(1)}\pm \pi\in [0,2\pi)$}
  }
  }
    \red{\KwLet{$\boldr^{(1)}\coloneqq A(\boldt^{(1)},\boldalpha^{(1)})^{-1}\begin{bmatrix}
      1\\
      \boldsymbol{0}
  \end{bmatrix}$.}}
  \tcp*{This will produce admissible parameters}
  \For{$\nu = 1,2,\cdots$}{
    \KwCompute{$\begin{bmatrix}
      h & \boldlambda
  \end{bmatrix}^T = \red{[A(\boldt,\boldalpha)^T]}^{-1}c_f(\boldt,\boldalpha)$}
  \tcp*{get trial coefficients $\boldlambda$}
  \KwLet{$\varphi \coloneqq \sum_{k=1}^{n}\lambda_k\varphi_k$}
  
  \If{$\|f-\varphi\|_I - h<\mathrm{threshold}\cdot h$}
    {
        \KwReturn{$\{t_j\}_{j=1}^{n+1}$, $\{\alpha_j\}_{j=1}^{n+1}$, $\{r_j\}_{j=1}^{n+1}$, $\{\lambda_j\}_{j=1}^{n}$} \tcp*{best approximation found}
    }
  
    \Else
    {
        \KwFind{$x\in [0,1]$ and $\vartheta\in [0,2\pi)$ satisfying $f(x)-\varphi(x) = e^{i\vartheta}\|f-\varphi\|_I$}
        \KwLet{$\boldd$ be \red{the} solution to $\boldd\cdot A(\boldt,\boldalpha)^T = \begin{bmatrix}
    1 & 
    \re(e^{-i\vartheta}\varphi_1(x)) &
    \cdots & 
    \re(e^{-i\vartheta}\varphi_n(x))
\end{bmatrix}$}
        \KwFind{$\rho = \operatorname{argmin}\{r_k/d_k:d_k>0\}$}
        \red{\KwLet{$\delta \coloneqq r_\rho/d_\rho$}}
        \KwLet{$t_\rho \coloneqq x$}
        \KwLet{$\alpha_\rho \coloneqq \vartheta$}
        \For{$j \in \{1\cdots n+1\}\setminus\{\rho\}$}{
        $r_j\coloneqq r_j-\delta\cdot d_j$
        }
        $r_\rho \coloneqq \delta$
    	
     %\KwReturn{$\{t_j\}_{j=1}^{n+1}$, $\{\alpha_j\}_{j=1}^{n+1}$, $\{r_j\}_{j=1}^{n+1}$, $\{\lambda_j\}_{j=1}^{n}$}
    }
  }
\caption{Remez Algorithm}
\label{alg:remez_algorithm}
\end{algorithm}

Algorithm \ref{alg:remez_algorithm} enables the computation of best approximations to complex valued functions.


\newpage

\begin{thebibliography}{10}

\bibitem{abdullayev-savchuk-tunc20}
F.~G. Abdullayev, V.~V. Savchuk, and T.~Tun\c{c}.
\newblock Exact estimates for {F}aber polynomials and for norm of {F}aber operator.
\newblock {\em Complex Var. Elliptic Equ.}, 65:293--305, 2020.

\bibitem{achieser56}
N.~I. Achieser.
\newblock {\em Theory of approximation}.
\newblock Frederick Ungar publishing co. New York, 1956.

\bibitem{alpan-goncharov-simsek18}
G.~Alpan, A.~Goncharov, and A.~N. \c{S}im\c{s}ek.
\newblock Asymptotic properties of {J}acobi matrices for a family of fractal measures.
\newblock {\em Exp. Math.}, 27:10--21, 2018.

\bibitem{andrievskii-blatt01}
V.~Andrievskii and H.-P. Blatt.
\newblock {\em Discrepancy of Signed Measures and Polynomial Approximation}.
\newblock Springer New York, 2001.

\bibitem{andrievskii-nazarov19}
V.~Andrievskii and F.~Nazarov.
\newblock On the {T}otik--{W}idom property for a quasidisk.
\newblock {\em Constr. Approx.}, 50:497--505, 2019.

\bibitem{bergman-rubin24}
A.~Bergman and O.~Rubin.
\newblock Chebyshev polynomials corresponding to a vanishing weight.
\newblock {\em J. Approx. Theory}, 301:106048, 2024.

\bibitem{chebyshev54}
P.~L. Chebyshev.
\newblock Th\'{e}orie des m\'{e}canismes connus sous le nom de parall\'{e}logrammes.
\newblock {\em M\'{e}m. des sav. \'{e}tr. pr\'{e}s. \`{a} l'Acad. de. St. P\'{e}tersb.}, 7:539--568, 1854.

\bibitem{cheney66}
E.~W. Cheney.
\newblock {\em Introduction to approximation theory}.
\newblock McGraw-Hill book company, 1966.

\bibitem{christiansen-eichinger-rubin24}
J.~S. Christiansen, B.~Eichinger, and O.~Rubin.
\newblock Extremal polynomials and sets of minimal capacity.
\newblock {\em Constr. Approx.}, 2024.

\bibitem{christiansen-simon-zinchenko-I}
J.~S. Christiansen, B.~Simon, and M.~Zinchenko.
\newblock Asymptotics of {C}hebyshev polynomials, {I}. subsets of {$\R$}.
\newblock {\em Invent. Math.}, 208:217--245, 2017.

\bibitem{christiansen-simon-zinchenko-III}
J.~S. Christiansen, B.~Simon, and M.~Zinchenko.
\newblock Asymptotics of {C}hebyshev polynomials, {III}. sets saturating {S}zeg{\H{o}}, {S}chiefermayr, and {T}otik--{W}idom bounds.
\newblock {\em Oper. Theory Adv. Appl.}, 276:231--246, 2020.

\bibitem{christiansen-simon-zinchenko-IV}
J.~S. Christiansen, B.~Simon, and M.~Zinchenko.
\newblock Asymptotics of {C}hebyshev polynomials, {IV}. {C}omments on the complex case.
\newblock {\em J. Anal. Math.}, 141:207--223, 2020.

\bibitem{christiansen-simon-zinchenko-V}
J.~S. Christiansen, B.~Simon, and M.~Zinchenko.
\newblock Asymptotics of {C}hebyshev polynomials, {V}. {R}esidual polynomials.
\newblock {\em Ramanujan J.}, 61:251--278, 2023.

\bibitem{clunie59}
J.~Clunie.
\newblock On {S}chlicht functions.
\newblock {\em Ann. Math.}, 69:511--519, 1959.

\bibitem{faber20}
G.~Faber.
\newblock \"{U}ber {T}schebyscheffsche polynome.
\newblock {\em J. Reine Angew. Math.}, 150:79--106, 1919.

\bibitem{faber-liesen-tichy10}
V.~Faber, J.~Liesen, and P.~Tich{\'{y}}.
\newblock On {C}hebyshev polynomials of matrices.
\newblock {\em SIAM J. Matrix Anal. Appl.}, 31:2205--2221, 2010.

\bibitem{fischer-modersitzki-93}
B.~Fischer and J.~Modersitzki.
\newblock An algorithm for complex linear approximation based on semi-infinite programming.
\newblock {\em Numer. Algor.}, 5:287--297, 1993.

\bibitem{foucart-lasserre19}
S.~Foucart and J.~B. Lasserre.
\newblock Computation of {C}hebyshev polynomials for union of intervals.
\newblock {\em Comput. Methods Funct. Theory.}, 19:625--641, 2019.

\bibitem{gaier99}
D.~Gaier.
\newblock The {F}aber operator and its boundedness.
\newblock {\em J. Approx. Theory}, 101:265--277, 1999.

\bibitem{goncharov-hatinoglu15}
A.~Goncharov and B.~Hatino{\u{g}}lu.
\newblock Widom factors.
\newblock {\em Potential Anal.}, 42:671--680, 2015.

\bibitem{greenbaum-trefethen-94}
A.~Greenbaum and L.~N. Trefethen.
\newblock {GMRES/CR and Arnoldi/Lanczos} as matrix approximation problems.
\newblock {\em SIAM J. Sci. Comput.}, 15:359--368, 1994.

\bibitem{grothkopf-opfer82}
U.~Grothkopf and G.~Opfer.
\newblock Complex {C}hebyshev polynomials on circular sectors with degree six or less.
\newblock {\em Math. Comput.}, 39:599--615, 1982.

\bibitem{he94}
M.~He.
\newblock The {F}aber polynomials for $m$-fold symmetric domains.
\newblock {\em J. Comput. Appl. Math.}, 54:313--324, 1994.

\bibitem{he-94-2}
M.~He.
\newblock Numerical results on the zeros of {F}aber polynomials for $m$-fold symmetric domains.
\newblock In {\em Exploiting symmetry in applied and numerical analysis, \emph{Lectures in Appl. Math.}}, pages 229--240. Amer. Math. Soc., 1994.

\bibitem{he96}
M.~He.
\newblock Explicit representations of {F}aber polynomials for $m$-cusped hypocycloids.
\newblock {\em J. Approx. Theory}, 87:137--147, 1996.

\bibitem{he-saff94}
M.~He and E.~B. Saff.
\newblock The zeros of {F}aber polynomials for an $m$-cusped hypocycloid.
\newblock {\em J. Approx. Theory}, 78:410--432, 1994.

\bibitem{iske18}
A.~Iske.
\newblock {\em Approximation theory and algorithms for data analysis}.
\newblock Springer--Verlag GmbH, Heidelberg, 2018.

\bibitem{kamo-borodin94}
S.~O. Kamo and P.~A. Borodin.
\newblock Chebyshev polynomials for {J}ulia sets.
\newblock {\em Moscow Univ. Math. Bull.}, 49:44--45, 1994.

\bibitem{komodromos-russell-tang95}
M.~Z. Komodromos, S.~F. Russell, and P.~T.~P. Tang.
\newblock Design of {FIR} filters with complex desired frequency response using a generalized {R}emez algorithm.
\newblock {\em IEEE Trans. Circuits Syst. II}, 42:274--278, 1995.

\bibitem{kovari-pommerenke67}
T.~K{\"o}vari and Ch. Pommerenke.
\newblock On {F}aber polynomials and {F}aber expansions.
\newblock {\em Math. Z.}, 99:193--206, 1967.

\bibitem{kruger-simon15}
H.~Kr{\"{u}}ger and B.~Simon.
\newblock Cantor polynomials and some related classes of {OPRL}.
\newblock {\em J. Approx. Theory}, 191:71--93, 2015.

\bibitem{kuijlaars-saff95}
A.~Kuijlaars and E.~B. Saff.
\newblock Asymptotic distribution of the zeros of {F}aber polynomials.
\newblock {\em Math. Proc. Camb. Phil. Soc.}, 118:437--447, 1995.

\bibitem{lawson61}
C.~L. Lawson.
\newblock {\em Contributions to the Theory of Linear Least Maximum Approximations}.
\newblock PhD thesis, UCLA, 1961.

\bibitem{lax02}
P.~D. Lax.
\newblock {\em Functional analysis}.
\newblock John Wiley \& Sons, Inc., 2002.

\bibitem{minadiaz06}
E.~{Mi{\~{n}}a-D\'{i}az}.
\newblock {\em Asymptotics for {F}aber polynomials and polynomials orthogonal over regions in the complex plane}.
\newblock PhD thesis, Vanderbilt University, 2006.

\bibitem{novello-schiefermayr-zinchenko21}
G.~Novello, K.~Schiefermayr, and M.~Zinchenko.
\newblock Weighted {C}hebyshev polynomials on compact subsets of the complex plane.
\newblock In {\em From operator theory to orthogonal polynomials, combinatorics, and number theory}, pages 357--370. Springer, 2021.

\bibitem{opfer76}
G.~Opfer.
\newblock An algorithm for the construction of best approximations based on {K}olmogorov's criterion.
\newblock {\em J. Approx. Theory}, 23:299--317, 1978.

\bibitem{ransford95}
T.~Ransford.
\newblock {\em Potential theory in the complex plane}.
\newblock Cambridge university press, 1995.

\bibitem{remez34-1}
E.~Remez.
\newblock Sur la d\'{e}termination des polyn\^{o}mes d'approximation de degr\'{e} donn\'{e}e.
\newblock {\em Comm. Soc. Math. Kharkov}, 10:41--63, 1934.

\bibitem{remez34-2}
E.~Remez.
\newblock Sur un proc\'{e}d\'{e} convergent d'approximations successives pour d\'{e}terminer les polyn\^{o}mes d'approximation.
\newblock {\em Compt. Rend. Acad. Sc.}, 198:2063--2065, 1934.

\bibitem{rivlin90}
T.~J. Rivlin.
\newblock {\em Chebyshev polynomials: from approximation theory to algebra and number theory}.
\newblock John Wiley \& sons, Inc., New York, 1990.

\bibitem{rubin24}
O.~Rubin.
\newblock Chebyshev polynomials in the complex plane and on the real line.
\newblock {\em \normalfont{\bf{ar{X}iv}}}, 2411.14175, 2024.

\bibitem{saff-stylianopoulos08}
E.~B. Saff and N.~S. Stylianopoulos.
\newblock Asymptotics for polynomial zeros: Beware of predictions from plots.
\newblock {\em Comput. Methods Funct. Theory.}, 8:385--407, 2008.

\bibitem{saff-stylianopoulos15}
E.~B. Saff and N.~S. Stylianopoulos.
\newblock On the zeros of asymptotically extremal polynomial sequences in the plane.
\newblock {\em J. Approx. Theory}, 191:118--127, 2015.

\bibitem{saff-totik90}
E.~B. Saff and V.~Totik.
\newblock Zeros of {C}hebyshev polynomials associated with a compact set in the plane.
\newblock {\em SIAM J. Math. Anal.}, 21:799--802, 1990.

\bibitem{smirnov-lebedev68}
V.~I. Smirnov and N.~A. Lebedev.
\newblock {\em Functions of a complex variable: constructive theory}.
\newblock The M.I.T. press, Massachusetts Institute of Technology, Cambridge, Massachusetts, 1968.

\bibitem{suetin74}
P.~K. Suetin.
\newblock Polynomials orthogonal over a region and bieberbach polynomials.
\newblock In {\em Proceedings of the Steklov institute of mathematics}. American mathematical society, Providence, Rhode Island, 1974.

\bibitem{suetin84}
P.~.K Suetin.
\newblock {\em Series of Faber polynomials}.
\newblock Gordon and Breach science publishers, Amsterdam, The Netherlands, 1998.

\bibitem{szego24}
G.~Szeg\H{o}.
\newblock {Bemerkungen zu einer Arbeit von Herrn M. Fekete: \"{U}ber die Verteilung der Wurzeln bei gewissen algebraischen Gleichungen mit ganzzahligen Koeffizienten}.
\newblock {\em Math. Z.}, 21:203--208, 1924.

\bibitem{tang87}
P.~T.~P. Tang.
\newblock {\em Chebyshev approximation on the complex plane}.
\newblock PhD thesis, University of California at Berkeley, 1987.

\bibitem{tang88}
P.~T.~P. Tang.
\newblock A fast algorithm for linear complex {C}hebyshev approximation.
\newblock {\em Math. Comp.}, 51:721--739, 1988.

\bibitem{thiran-detaille91}
J.-P. Thiran and C.~Detaille.
\newblock Chebyshev polynomials on circular arcs and in the complex plane.
\newblock In {\em Progress in Approximation Theory}, pages 771--786. Academic Press, Boston, MA, 1991.

\bibitem{toh-trefethen98}
K-.C. Toh and L.~N. Trefethen.
\newblock The {C}hebyshev polynomial of a matrix.
\newblock {\em SIAM J. Matrix Anal. Appl.}, 20:400--419, 1998.

\bibitem{totik14}
V.~Totik.
\newblock Chebyshev polynomials on compact sets.
\newblock {\em Potential Anal.}, 40:511--524, 2014.

\bibitem{totik-varga15}
V.~Totik and T.~Varga.
\newblock Chebyshev and fast decreasing polynomials.
\newblock {\em Proc. Lond. Math. Soc.}, 110:1057--1098, 2015.

\bibitem{ullman60}
J.~L. Ullman.
\newblock Studies in {F}aber polynomials. {I}.
\newblock {\em Trans. Am. Math. Soc.}, 94:515--528, 1960.

\bibitem{widom69}
H.~Widom.
\newblock Extremal polynomials associated with a system of curves in the complex plane.
\newblock {\em Adv. in Math.}, 3:127--232, 1969.

\end{thebibliography}
\end{document}